\documentclass[number,sort&compress,3p]{elsarticle}

\usepackage[english]{babel}

\usepackage{dsfont} 
\usepackage{mathrsfs} 
\usepackage{amsmath,amssymb}
\usepackage{verbatim} 
\usepackage{amsthm} 
\usepackage{graphicx} 
\usepackage[usenames]{color}
\usepackage{url}
\usepackage{float}

\def\1{\mathds{1}}

\def\NN{\mathbb N}

\def\RR{\mathbb R}
\def\CC{\mathbb C}

\def\fcar{\mathds{1}}

\renewcommand\Re{\operatorname{Re}}

\def\diag{\textrm{diag}}
\def\ii{\text{\rm i}}

\def\esp{\mathbf E}
\def\var{\mathbf{Var}}

\def\prob{\mathbf P}
\def\calN{\mathcal N}
\def\simiid{\overset{iid}{\sim}}
\def\convloi{\xrightarrow[\sigma_*\to 0]{\mathscr D}}
\def\convproba{\xrightarrow{P}}

\def\Nzeroun{\mathcal{N}(0,1)}

\theoremstyle{plain}
\newtheorem{theorem}{Theorem} 
\newtheorem{lemma}{Lemma}
\newtheorem{proposition}{Proposition}
\newtheorem*{theorem*}{Theorem} 
\newtheorem*{lemma*}{Lemma}
\newtheorem*{proposition*}{Proposition}

\theoremstyle{remark}
\newtheorem{corollary}{Corollary}
\newtheorem{remark}{Remark}
\newtheorem*{remark*}{Remark}

\newtheorem*{note*}{Note}

\theoremstyle{definition}

\newtheorem*{definition*}{Definition}

\def\oomega{\boldsymbol{\omega}}
\def\bnu{\boldsymbol{\nu}}
\def\bY{\boldsymbol{Y}}
\def\bZ{\boldsymbol{Z}}
\def\bX{\boldsymbol{X}}
\def\bB{\boldsymbol{B}}
\def\bS{\boldsymbol{S}}
\def\boldf{\boldsymbol{f}}
\def\ee{\boldsymbol{e}}
\def\bu{\boldsymbol{u}}
\def\by{\boldsymbol{y}}
\def\bc{\boldsymbol{c}}
\def\bs{\boldsymbol{s}}
\def\ssigma{\boldsymbol{\sigma}}
\def\nnu{\boldsymbol{\nu}}
\def\eps{\epsilon}
\def\eeps{\boldsymbol{\epsilon}}

\def\diese{\text{\tt\#}}
\def\etadiese{\eta^{\diese}}
\def\sigmadiese{\sigma^{\diese}}
\def\Ydiese{Y^{\diese}}
\def\bYdiese{{\boldsymbol{Y}^{\diese}}}
\def\epsdiese{\epsilon^{\diese}}

\def\cdiese{c^\diese}

\def\bcdiese{\boldsymbol{c}^\diese}

\def\calF{\mathcal{F}}

\usepackage{numcompress}
\begin{document}

\begin{frontmatter}

\title{Curve registration by nonparametric goodness-of-fit testing}

\author{Olivier Collier and Arnak S. Dalalyan}
\address{ENSAE ParisTech/CREST/GENES\\ 3 avenue P.\ Larousse\\ 92245 Malakoff, FRANCE}

\begin{abstract}
The problem of curve registration appears in many different areas of applications ranging from
neuroscience to road traffic modeling. In the present work\footnote{This paper was presented in part at the AI-STATS 2012 conference.},
we propose a nonparametric testing framework in which
we develop a generalized likelihood ratio test to perform curve registration. We first prove that,
under the null hypothesis,  the resulting test statistic is asymptotically distributed as a chi-squared random
variable. This result, often referred to as Wilks' phenomenon, provides a natural threshold for the test of a
prescribed asymptotic significance level and a natural measure of lack-of-fit in terms of the $p$-value of the $\chi^2$-test.
We also prove that the proposed test is consistent, \textit{i.e.}, its power is asymptotically equal
to $1$. Finite sample properties of the proposed methodology are demonstrated by numerical simulations. As an application,
a new local descriptor for digital images is introduced and an experimental evaluation of its discriminative power is conducted.
\end{abstract}


\end{frontmatter}

\section*{Introduction}
Boosted by applications in different areas such as biology, medicine, computer vision and road traffic forecasting, the problem
of curve registration and, more particularly, some aspects of this problem related to nonparametric and semiparametric estimation,
have been explored in a number of recent statistical studies.
In this context, the model used for deriving statistical inference represents the input data as a finite collection of noisy signals such that
each input signal is obtained from a given signal, termed mean template or structural pattern, by a parametric deformation and by
adding a white noise. Hereafter, we refer to this as the \textit{deformed mean template} model. The main difficulties for
developing statistical inference in this problem are caused by the nonlinearity of the deformations and the fact that not only the
deformations but also the mean template used to generate the observed data are unknown.

While the problems of estimating the mean template and the deformations was thoroughly investigated in
recent years, the question of the adequacy of modeling the available data by the deformed mean template model received little
attention. By the present work, we intend to fill this gap by introducing a nonparametric goodness-of-fit
testing framework that allows us to propose a measure of appropriateness of a deformed mean template model. To this end, we focus
our attention on the case where the only allowed deformations are translations and propose a measure of
goodness-of-fit based on the $p$-value of a chi-squared test.

\subsection*{Model description}

We consider the case of functional data, that is each observation is a function on a fixed interval, taken for simplicity
equal to $[0,1]$. More precisely, assume that two independent samples, denoted $\{X_i\}_{i=1,\ldots,n}$ and $\{X_i^\diese\}_{i=1,\ldots,n^\diese}$,
of functional data are available such that within each sample the observations are independent identically distributed (i.i.d.\ ) drifted and scaled
Brownian motions. Let $f$ and $f^\diese$ be the corresponding drift functions: $f(t)={d\esp[X_1(t)]}/{dt}$ and $f^\diese(t)={d\esp[X_1^\diese(t)]}/{dt}$.
Then, for $t\in[0,1]$,
$$
X_i(t)=\int_0^t f(u)\,du+sB_i(t),\quad X_\ell^\diese(t)=\int_0^t f^\diese(u)\,du+s^\diese B_\ell^\diese(t),
$$
where $s,s^\diese>0$ are the scaling parameters and  $(B_1,\ldots,B_n,B_1^\diese,\ldots,B_{n^\diese}^\diese)$ are independent Brownian motions.
Since we assume that the entire paths are observed, the scale parameters $s$ and $s^\diese$
can be recovered with arbitrarily small error using the quadratic variation. So, in what follows,
these parameters are assumed to be known (an extension to the setting of unknown noise level is briefly discussed in
Section~\ref{sec:ext}).

The goal of the present work is to provide a statistical testing procedure for deciding
whether the curves of the functions $f$ and $f^\diese$ coincide up to a translation. Considering periodic extensions of $f$ and $f^\diese$ on the whole real
line, this is equivalent to checking the null hypothesis
\begin{equation}\label{eq:3}
\mathbf{H_0}:\qquad \exists\ (\tau^*,a^*)\in [0,1]\times\RR\qquad \text{ such that\ }\qquad f(\cdot)=f^\diese(\cdot+\tau^*)+a^*.
\end{equation}
If the null hypothesis is satisfied, we are in the set-up of a deformed mean template model, where $f(\cdot)$ plays the role of the mean template and
spatial translations represent the set of possible deformations.

Starting from~\cite{Golubev88} and \cite{Kneip92}, semiparametric and nonparametric estimation in different
instances of the deformed mean template model have been intensively investigated  \cite{Ronn01,Dalalyan06,Gamboa07,Dalalyan08,Castillo09,BigotGadat10,TriganoRitov,
Castillo11,Hardle90,Carroll92,Vimond10,Castillo07,BigotVim09} with applications to image warping \citep{Glasbey01,Bigotetal09}. However, prior
to estimating the common template, the deformations or any other object involved in a deformed mean template model, it is natural to check its appropriateness, which is the
purpose of this work.

To achieve this goal, we first note that the pair of sequences of complex-valued random variables
$\bY=(Y_0,Y_1,\ldots)$ and $\bY^\diese=(Y_0^\diese,Y_1^\diese,\ldots)$, defined by
\begin{align*}
\big[Y_j,Y_j^\diese\big]=\int_0^1 e^{2\pi \ii jt} \,d\bigg[\frac1{n}\sum_{i=1}^n X_i(t), \frac1{n^\diese}\sum_{\ell=1}^{n^\diese} X_\ell^\diese(t)\bigg],
\end{align*}
is a sufficient statistic in the model generated by observations $(X_1,\ldots,X_n)$ and  $(X_1^\diese,\ldots,X_{n^\diese}^\diese)$. Therefore,
without any loss of information, the initial (functional) data can be replaced by the transformed data $(\bY,\bY^\diese)$. Let us
denote by $c_j=\int_0^1 f(x)\,e^{2\ii j\pi x }\,dx$ and  $\cdiese_j=\int_0^1 f^\diese(x)\,e^{2\ii j\pi x }\,dx$ the complex Fourier
coefficients of the signals $f$ and $f^\diese$. Then, the first components of the observed sequences, $(Y_0,Y^\diese_0)$, can be written as
\begin{equation*}
Y_0 = c_0 + \frac{s}{\sqrt{n}}\;\epsilon_0,\quad \Ydiese_0 = \cdiese_0 + \frac{s^\diese}{\sqrt{n^\diese}}\;\epsdiese_0,
\end{equation*}
where $\epsilon_0$ and $\epsdiese_0$ are two independent, real, standard Gaussian variables.
Furthermore, for $j\ge 1$, we have
\begin{equation}\label{eq:5}
Y_j = c_j + \frac{s}{\sqrt{2n}}\;\epsilon_j,\quad \Ydiese_j = \cdiese_j + \frac{s^\diese}{\sqrt{2n^\diese}}\;\epsdiese_j,
\end{equation}
where the complex valued random variables $\epsilon_j$, $\epsdiese_j$ are i.i.d.\ standard Gaussian: $\epsilon_j, \epsdiese_j
\sim\mathcal N_{\mathbb C}(0,1)$, which means that their real and imaginary parts are independent $\mathcal N(0,1)$ random variables.
Moreover, $(\epsilon_0, \epsdiese_0)$ are independent of $\{(\epsilon_j,\epsdiese_j):j\ge1\}$.
In what follows, we will use boldface letters for denoting
vectors or infinite sequences so that, for example, $\bc$ and $\bcdiese$ refer to $\{c_j; j=0,1,\ldots\}$ and
$\{\cdiese_j; j=0,1,\ldots\}$, respectively.

Under the mild assumption that $f$ and $f^\diese$ are squared integrable, the likelihood ratio of the Gaussian process $\bY^{\bullet,\diese}=(\bY,\bY^\diese)$
is well defined. Using the notation $\bc^{\bullet,\diese}=(\bc,\bc^\diese)$, $\sigma=s/\sqrt{2n}$ and $\sigma^\diese=s^\diese/\sqrt{2n^\diese\!}$, the corresponding negative log-likelihood is given by
\begin{equation}\label{eq:8}
\ell(\bY^{\bullet,\diese},\bc^{\bullet,\diese})
	= \frac{(Y_0-c_0)^2}{4\sigma^2}+\frac{(Y^\diese_0-c^\diese_0)^2}{4\sigmadiese{}^2}+
		\sum_{j\ge 1}\bigg( \frac{|Y_j-c_j|^2}{2\sigma^2}+\frac{|Y_j^\diese-c^\diese_j|^2}{2\sigmadiese{}^2}\bigg).
\end{equation}
In the present work, we present a theoretical analysis of the penalized likelihood ratio test  in the asymptotics of large samples, \textit{i.e.}, when both $n$ and $n^\diese$
tend to infinity, or equivalently, when $\sigma$ and $\sigma^\diese$ tend to zero. The finite sample properties are examined through numerical simulations.

\subsection*{Some motivations}
Even if the shifted curve model is a very particular instance of the general deformed mean template model, it plays a central role in several
applications.  To cite a few of them:
\parskip=-2pt
\begin{description}\itemsep=2pt
\item[ECG interpretation:]An~electro-cardiogram~(ECG) can be seen as a collection of replica of nearly the same signal, up to a time shift.
Significant information about heart malformations or diseases can be extracted from the mean signal if we are able to align the available curves.
For more details we refer to \cite{TriganoRitov}.
\item[Road traffic forecast:] In \cite{Loubes06}, a road traffic forecasting procedure is introduced. To this end, archetypes of the different
types of road trafficking behavior on the Parisian highway network are built, using a hierarchical classification method. In each obtained
cluster, the curves all represent the same events, only randomly shifted in time.
\item[Keypoint matching:] An important problem in computer vision is to decide whether two points in a same image or in two different images
correspond to the same real-world point. The points in images are then usually described by the regression function of the magnitude of the
gradient over the direction of the gradient of the image  restricted to a given neighborhood (cf.\ \cite{Lowe}). The
methodology we shall develop in the present paper allows to test whether two points in images coincide, up to a rotation and an illumination
change, since a rotation corresponds to shifting the argument of the regression function by the angle of the rotation.
\end{description}
\parskip=3pt
\vspace{-10pt}

\subsection*{Relation to previous work}

The problem of estimating the parameters of the deformation is a semiparametric one, since the deformation involves a finite number of parameters that
have to be estimated by assuming that the unknown mean template is merely a nuisance parameter. In contrast, the testing problem we are concerned with
is clearly nonparametric. The parameter describing the probability distribution of the observations is infinite-dimensional not only under the alternative
but also under the null  hypothesis. Surprisingly, the statistical literature on this type of testing problems is very scarce. Indeed, while \cite{Horowitz01}
analyzes the optimality and the adaptivity of testing procedures in the setting of a parametric null hypothesis against a nonparametric alternative,
to the best of our knowledge, the only papers concerned with nonparametric null hypotheses are \citep{Baraudetal03,Baraudetal05} and \citep{Pouet05}.
Unfortunately, the results derived in \citep{Baraudetal03,Baraudetal05} are inapplicable in our set-up since the null hypothesis in our problem is neither linear nor convex. The set-up of \citep{Pouet05} is closer to ours. However, they only investigate the minimax rates of separation without providing the asymptotic distribution of the proposed test statistic, which generally results in an overly conservative testing procedure.
Furthermore, their theoretical framework comprises a condition on the sup-norm-entropy of the null hypothesis, which is irrelevant in our set-up and may be violated.

There is also a relatively vast literature on the nonparametric comparison of several curves (see \citep{Munk_Dette,Neumeyer,PardoFernandez,Srihera} and the references therein). The results developed in most of these papers concern the regression model with random design and assume that a noisy version of each curve is observed. The tests proposed therein are mainly based on kernel smoothing, which are quite different from the tests analyzed in the present paper. In particular, tests based on kernel smoothing may achieve optimal rates only for smoothness classes with regularity less than $1$. Note also that when transposed to the model of Gaussian white noise, the problem of testing the equality of two functions boils down
to that of testing that a function vanishes everywhere, just by computing the difference of observed noisy signals. Thus the null hypothesis becomes
simple.

In a companion\footnote{The writing of the paper \cite{Collier2012} being completed slightly later than the present one, it has been published earlier because of the randomness of the reviewing process.} paper \citep{Collier2012}, the problem of curve registration by statistical hypotheses testing
has been tackled from an asymptotic minimax point of view. In \citep{Collier2012}, as a complement of the present work,  minimax rates of separation
(up to $\log$ factors) are established and a smoothness-adaptive test is proposed under some assumptions, which are substantially stronger than those
required in this paper. Note also that preliminary versions of some results reported in the present manuscript have been presented at AI-STATS conference \citep{jmlr_DalalyanC12}. The present manuscript is a corrected, completed and significantly developed version of \citep{jmlr_DalalyanC12}. More precisely, Section~\ref{sec:ext} contains important extensions of the presented methodology to the case of multidimensional signals
and unknown noise variance, whereas Section~\ref{sec:LoFT} presents an original application to the problem of keypoint matching in computer vision.

\subsection*{Our contribution}

We adopt, in this work, the approach based on the generalized likelihood ratio tests, cf.~\cite{Fan07} for a comprehensive account on the topic. The advantage
of this approach is that it provides a general framework for constructing testing procedures which asymptotically achieve the prescribed significance level for the
type I error and, under mild conditions, have a power that tends to one. It is worth mentioning that in the context of nonparametric testing, the use  of the
\textit{generalized} likelihood ratio leads to a substantial improvement upon the likelihood ratio, very popular in parametric statistics. In simple words,
the generalized likelihood allows to incorporate some prior information on the unknown signal in the test statistic which introduces more flexibility and
turns out to be crucial both in theory and in practice, see \cite{FanZhang01}.

We prove that under the null hypothesis the generalized likelihood ratio test statistic is asymptotically distributed as a $\chi^2$-random variable. This allows
us to choose a threshold that makes it possible to asymptotically control the test significance level without being excessively conservative. Such results
are referred to as Wilks' phenomena. In this relation, let us quote \cite{Fan07}: ``While we have observed the Wilks' phenomenon and demonstrated it
for a few useful cases, it is impossible for us to verify the phenomenon for all nonparametric hypothesis testing problems. The Wilks' phenomenon
needs to be checked for other problems that have not been covered in this paper. In addition, most of the topics outlined in the above discussion
remains open and are technically and intellectually challenging. More developments are needed, which will push the core of statistical theory and
methods forward.''

It is noteworthy that our results apply to the Gaussian sequence model (\ref{eq:5}), which is often seen as a prototype of nonparametric
statistical model. In fact, it is provably asymptotically equivalent to many other statistical models \citep{BrownLow96,Nussbaum96,GramaNussbaum,DR07,Reiss}
and captures most theoretical difficulties of the statistical inference. Furthermore, using the aforementioned results on asymptotic equivalence,
the main theoretical findings of the present paper automatically carry over the nonparametric regression model, the density model, the ergodic diffusion model,
\textit{etc}.

Finally, we provide a detailed explanation of how the proposed methodology can be used for solving the problem of keypoint matching in digital images. This leads
to a new descriptor termed Localized Fourier Transform which is particularly well adapted for testing for rotation. The first experiments reported in this
work show the validity of our theoretical findings and the potential of the new descriptor.

\subsection*{Organization}

The rest of the paper is organized as follows. After a brief presentation of the model, we introduce the generalized likelihood ratio framework in Section~\ref{sec:1}.
The main results characterizing the asymptotic behavior of the proposed testing procedure, based on generalized likelihood ratio testing for a large variety of shrinkage
weights, are stated in Section~\ref{sec:2}. Section~\ref{sec:ext} contains extensions of our results to the multidimensional setting and to the case of unknown noise
magnitude. Some numerical examples illustrating the theoretical results are included in Section~\ref{sec:3}, while Section~\ref{sec:LoFT} is devoted to the application
of the proposed methodology to the problem of keypoint matching in computer vision. The resulting Localized Fourier Transform (LoFT) descriptor is tested on a pair
of real images degraded by white Gaussian noise. The proofs of the lemmas and of the theorems are postponed to the Appendix.


\section{Penalized Likelihood Ratio Test}\label{sec:1}

We are interested in testing the hypothesis (\ref{eq:3}), which translates in the Fourier domain to
\begin{equation*}
\mathbf{H_0} :\  \text{there exists} \quad \bar\tau^* \in [0,2\pi[\quad \text{such that}\quad
c_j = e^{-\ii j \bar\tau^*} \cdiese_j\quad \text{for all } j\ge 1.
\end{equation*}
Indeed, one easily checks that if (\ref{eq:3}) is true, then\footnote{We use here the change of the variable $z=t-\tau^*$ and the fact that the integral of a 1-periodic function on an
interval of length one does not depend on the interval of integration.}
$$
c_j^\diese=\int_0^1 f(t-\tau^*)e^{2\ii j\pi t}\,dt=e^{2\ii j\pi\tau\!^*} \int_0^1 f(z) e^{2\ii j\pi z}\,dz=e^{2\ii j\pi\tau\!^*} c_j
$$
and, therefore, the aforementioned relation holds with $\bar\tau^*= 2\pi\tau^*$. If no additional assumptions are imposed on the functions $f$ and $f^\diese$, or equivalently on their Fourier
coefficients $\bc$ and $\bcdiese$, the nonparametric testing problem has no consistent solution. A natural assumption widely used in
nonparametric statistics is that $\bc=(c_0,c_1,\ldots)$ and $\bcdiese=(\cdiese_0,\cdiese_1,\ldots)$ belong to some Sobolev ball
$$
\calF_{s,L}=\Big\{\bu=(u_0,u_1,\ldots): \sum_{j=0}^{+\infty} j^{2s} |u_j|^2 \leq L^2\Big\},
$$
where the positive real numbers $s$ and $L$ stand for the smoothness and the radius of the class $\calF_{s,L}$.

Since we will also aim at establishing the (uniform) consistency of the proposed testing procedure, we need to precise the
form of the alternative. It seems that the most compelling form for the null and the alternative is
\begin{equation}\label{eq:7}
\begin{cases}
\mathbf{H_0} :&\text{there exists} \, \bar\tau^*\in[0,2\pi[  \text{ such that }
c_j = e^{-\ii j \bar\tau^*} \cdiese_j\quad \text{for all } j\ge 1.\\
\mathbf{H_1} :& \inf_{\tau} \sum_{j=1}^{+\infty} |c_j-e^{-\ii j\tau}c_j^\diese|^2\ge \rho
\end{cases}
\end{equation}
for some $\rho>0$. In other terms, under $\mathbf{H_0}$ the graph of the function $f^\diese$ is obtained from that of $f$ by a translation.
To ease notation, we will use the symbol $\circ$ to denote coefficient-by-coefficient multiplication, also known as the Hadamard product,
and $\ee(\tau)$ will stand for the sequence $(e^{-\ii\tau},e^{-2\ii\tau},\ldots)$.

To present the penalized likelihood ratio test, which is a variant of the generalized likelihood ratio test, we introduce a penalization in terms of weighted
$\ell^2$-norm of $\bc^{\bullet,\diese}$. In this context, the choice of the $\ell^2$-norm penalization is mainly motivated by the fact that
Sobolev regularity assumptions are made on the functions $f$ and $f^\diese$. For a sequence of non-negative real numbers, $\oomega$, we
define the weighted $\ell_2$ norm $\|\bc\|_{\oomega,2}^2=\sum_{j\geq 0}\omega_j|c_j|^2$. We will also use the standard notation
$\|\bu\|_p=(\sum_{j} |u_j|^p)^{1/p}$ for any $p>0$. Using this notation, the penalized log-likelihood is given by
\begin{align}
p\ell(\bY^{\bullet,\diese},\bc^{\bullet,\diese})=& \ell(\bY^{\bullet,\diese},\bc^{\bullet,\diese})+
\frac{\|\bc\|_{\oomega,2}^2}{2\sigma^2}+
\frac{\|\bc^\diese\|_{\oomega,2}^2}{2\sigmadiese{}^2}\:.\label{eq:9}
\end{align}
The resulting penalized likelihood ratio test is based on the test statistic
\begin{align}
\Delta(\bY^{\bullet,\diese})=&\min_{\bc^{\bullet,\diese}:\mathbf{H_0}\text{ is true}}p\ell(\bY^{\bullet,\diese},\bc^{\bullet,\diese})-\min_{\bc^{\bullet,\diese}}p\ell(\bY^{\bullet,\diese},\bc^{\bullet,\diese}).\label{eq:10}
\end{align}
It is clear that $\Delta(\bY^{\bullet,\diese})$ is always non-negative. Furthermore, it is small when $\mathbf{H_0}$ is satisfied and is large
if $\mathbf{H_0}$ is violated. Therefore, $\Delta(\bY^{\bullet,\diese})$ is a good test statistic for deciding whether or not the null hypothesis $\mathbf{H_0}$ should be rejected.

\begin{lemma}\label{lem:0}
The test statistic $\Delta(\bY^{\bullet,\diese})$ can be written in the following form:
\begin{equation}\label{eq:11}
\Delta(\bY^{\bullet,\diese})=\frac{1}{2(\sigma^2+(\sigmadiese)^2)}\min_{\tau\in[0,2\pi]}
\sum_{j=1}^{+\infty} \frac{|Y_j-e^{-\ii j\tau}Y_j^\diese|^2}{1+\omega_j}.
\end{equation}
\end{lemma}

\begin{proof}
We start by noting that the minimization of the quadratic functional (\ref{eq:9}) leads to
\begin{equation}\label{eqq:1}
\min_{\bc^{\bullet,\diese}}p\ell(\bY^{\bullet,\diese},\bc^{\bullet,\diese})=\frac1{2\sigma^2}\sum_{j\geq1} \frac{\omega_j}{1+\omega_j}\,|Y_j|^2+
\frac1{2\sigmadiese{}^2}\sum_{j\geq1} \frac{\omega_j}{1+\omega_j}\,|Y^\diese_j|^2.
\end{equation}
Let us compute the statistic $\min_{\bc^{\bullet,\diese}:\mathbf{H_0}\text{ is true}}p\ell(\bY^{\bullet,\diese},\bc^{\bullet,\diese})$
that can be equivalently written in the form $\min_{\tau\in[0,2\pi]}\min_{c_j =
e^{-\ii j\bar\tau^*}c^\diese_j,j\ge1}p\ell(\bY^{\bullet,\diese},\bc^{\bullet,\diese})$.
When $\bc^{\bullet,\diese}$ satisfies the null hypothesis, simple algebra yields that the relation
$$
p\ell(\bY^{\bullet,\diese},\bc^{\bullet,\diese}) =
\sum_{j\ge 1} \bigg(\frac{|Y_j-c_j|^2}{2\sigma^2} +\frac{|e^{-\ii j \bar\tau^*}Y^\diese_j-c_j|^2}{2\sigma^\diese{}^2}\bigg)
+\frac{\|\bc\|_{\oomega,2}^2}{2\sigma^2}+\frac{\|\bc\|_{\oomega,2}^2}{2\sigma^\diese{}^2}
$$
is true for some $\bar\tau^*$. We need to compute the minimum with respect to $\bc$ of the right-hand side of the last
display. To ease notation, let us set $\bar\sigma^2 = 1/(\frac1{\sigma^2}+\frac1{\sigmadiese{}^2})$ and
$\bZ = \sigma^{-2}\bY+\sigmadiese{}^{-2}\ee(\bar\tau^*)\circ \bY^\diese$. Then, we get
$$
p\ell(\bY^{\bullet,\diese},\bc^{\bullet,\diese}) = \sum_{j\ge 1} \bigg(\frac{|Y_j|^2}{2\sigma^2} +\frac{|Y^\diese_j|^2}{2\sigma^\diese{}^2}+\frac{1+\omega_j}{2\bar\sigma^2}\;|c_j|^2 - \Re(\bar Z_j c_j)\bigg).
$$
The minimum of this expression is attained at $c_j = \frac{\bar\sigma^2 Z_j}{1+\omega_j}$. This leads to
\begin{equation}\label{eqq:2}
\min_{\bc^{\bullet,\diese}:\mathbf{H_0}\text{ is true}}p\ell(\bY^{\bullet,\diese},\bc^{\bullet,\diese})
= \sum_{j\ge 1} \bigg(\frac{|Y_j|^2}{2\sigma^2} +\frac{|Y^\diese_j|^2}{2\sigma^\diese{}^2}-\frac{\bar\sigma^2}{2(1+\omega_j)}\;|Z_j|^2\bigg).
\end{equation}
Combining equations (\ref{eqq:1}) and (\ref{eqq:2}) we get
$$
\Delta(\bY^{\bullet,\diese})=\frac{1}{2\sigma^2}\sum_j \frac{|Y_j|^2}{1+\omega_j}+
\frac{1}{2\sigmadiese{}^2}\sum_j \frac{|Y^\diese_j|^2}{1+\omega_j}-\sum_j\frac{\bar\sigma^2}{2(1+\omega_j)}\;|Z_j|^2.
$$
To complete the proof, it suffices to replace $\bZ$ by its definition and to use the identity
$\sigma^{-2}-\bar\sigma^2 \sigma^{-4} = \sigmadiese{}^{-2}-\bar\sigma^2 \sigmadiese{}^{-4} = (\sigma^2+\sigmadiese{}^2)^{-1}$.
\end{proof}
From now on, it will be more convenient to use the notation $\nu_j=1/(1+\omega_j)$. The elements of the sequence $\bnu=\{\nu_j;\ j\ge 1\}$
are hereafter referred to as shrinkage weights. They are allowed to take any value between $0$ and $1$. Even the value 0 will be
authorized, corresponding to the limiting case when $w_j=+\infty$, or equivalently to our belief that the corresponding Fourier coefficient is $0$. The test statistic can then be written as:
\begin{equation}\label{eq:12}
\Delta(\bY^{\bullet,\diese})=\frac{1}{2(\sigma^2+(\sigmadiese)^2)}\min_{\tau\in[0,2\pi]}
\|\bY-\ee(\tau)\circ \bY^\diese\|^2_{2,\bnu}\,,
\end{equation}
and one of the goals is to find the asymptotic distribution of this quantity under the null hypothesis.

\section{Main results}\label{sec:2}

The test based on the generalized likelihood ratio statistic involves a sequence $\bnu$, which should be chosen by the user. However, we are able to provide
theoretical guarantees only under some conditions on these weights. To state these conditions, we set $\sigma_*=\max(\sigma,\sigmadiese)$ and choose a positive
integer $N=N_{\sigma_*}\ge 2$, which represents the number of Fourier coefficients involved in our testing procedure. In addition to requiring that $0\le\nu_j\le 1$ for every $j$, we assume that:
\begin{align*}
~&\textbf{(A)}\quad\nu_1=1,\qquad \text{and} \qquad \nu_j = 0,\ \forall j > N_{\sigma_*}, \phantom{\sum\limits_{j\geq0}}\qquad\qquad~ \\[-10pt]
~&\textbf{(B)}\quad  \text{for some positive constant $\underline c$, it holds that } \sum\limits_{j\geq 1} \nu_j^2 \ge \underline cN_{\sigma_*}.\qquad\qquad
\end{align*}
Moreover, we will use the following condition in the proof of the consistency of the test:
\begin{equation*}
~\textbf{(C)} \quad \exists \, \overline c> 0, \,\text{such that } \min\{j\geq0,\nu_j<\overline c\} \to +\infty,\ \text{as } \sigma_*\to 0.\qquad\qquad
\end{equation*}
In simple words, this condition implies that the number of terms $\nu_j$ that are above a given strictly positive level goes to $+\infty$
as $\sigma_*$ converges to $0$. If $N_{\sigma_*}\to+\infty$ as $\sigma_*\to 0$, then all the aforementioned conditions are satisfied for the shrinkage
weights $\bnu$ of the form $\nu_{j+1}=h(j/N_{\sigma_*})$, where $h:\RR\to[0,1]$ is an integrable function, supported on $[0,1]$, continuous in $0$ and
satisfying $h(0)=1$. The classical examples of shrinkage weights include:
\begin{equation}\label{weights}
\nu_j =
\begin{cases}
\begin{tabular}{lll}
$\mathds{1}_{\{j \leq N_{\sigma_*}\}}$,& &(projection weight)\\[8pt]
$\big\{1+\big(\frac{j}{\kappa N_{\sigma_*}}\big)^\mu\big\}^{-1} \mathds{1}_{\{j \leq N_{\sigma_*}\}},$ & $\kappa>0$, $\mu > 1$, &  (Tikhonov weight)\\[8pt]
$\big\{1-\big(\frac{j}{N_{\sigma_*}}\big)^\mu\big\}_+,\,$ & $\mu > 0$. & (Pinsker weight)
\end{tabular}
\end{cases}
\end{equation}
Note that condition {\bf(C)} is satisfied in all these examples with $\overline c=0.5$, or any other value in $(0,1)$. Here on,
we write $\Delta_{\bnu,\sigma_*}(\bY^{\bullet,\diese})$ instead of $\Delta(\bY^{\bullet,\diese})$ in order to stress its dependence 
on $\bnu$ and on $\sigma_*$.

\begin{theorem}\label{th:1}
Let  $\bc\in\calF_{1,L}$ and $|c_1| >0$. Assume that the shrinkage weights $\nu_j$ are chosen to satisfy conditions {\bf(A)}, {\bf(B)}, $N_{\sigma_*}\to+\infty$ and $\sigma_*^2 N_{\sigma_*}^{5/2}\log(N_{\sigma_*})\to 0$. Then, under the null hypothesis, the test statistic $\Delta_{\bnu,\sigma_*}(\bY^{\bullet,\diese})$ is asymptotically
distributed as a Gaussian random variable:
\begin{equation}\label{conv}
\frac{\Delta_{\bnu,\sigma_*}(\bY^{\bullet,\diese})-\|\bnu\|_1}{\|\bnu\|_2}\convloi\mathcal N(0,1).
\end{equation}
\end{theorem}

The main outcome of this result is a test of hypothesis $\mathbf{H_0}$ that is asymptotically of a prescribed significance level $\alpha\in(0,1)$. Indeed,
let us define the test that rejects $\mathbf{H_0}$ if and only if
\begin{equation}\label{eq:13}
\Delta_{\bnu,\sigma_*}(\bY^{\bullet,\diese})\ge \|\bnu\|_1 +  z_{1-\alpha}\|\bnu\|_2,
\end{equation}
where $z_{1-\alpha}$ is the $(1-\alpha)$-quantile of the standard Gaussian distribution.

\begin{corollary}
The test of hypothesis $\mathbf{H_0}$ defined by the critical region (\ref{eq:13}) is asymptotically of significance level $\alpha$.
\end{corollary}

\begin{remark}
Let us consider the case of projection weights $\nu_j=\mathds{1}(j\le N_{\sigma_*})$.
One can reformulate the asymptotic relation stated in Theorem~\ref{th:1} by claiming that
$2\Delta_{\bnu,\sigma_*}(\bY^{\bullet,\diese})$ is approximately $\mathcal N(2N_{\sigma_*},4N_{\sigma_*})$
distributed. Since the latter distribution approaches the chi-squared distribution (as $N_{\sigma_*}\to\infty$),
we get:
$$
2\Delta_{\bnu,\sigma_*}(\bY^{\bullet,\diese})\ \stackrel{\mathcal D}{\approx}\ \chi^2_{2N_{\sigma_*}}\,,\qquad\text{as}\qquad \sigma_*\to 0.
$$
In the case of general shrinkage weights satisfying the assumptions stated in the beginning of this section, an analogous
relation holds as well:
$$
\frac{2\|\bnu\|_1}{\|\bnu\|_2^2}\;\Delta_{\bnu,\sigma_*}(\bY^{\bullet,\diese})\ \stackrel{\mathcal D}{\approx}\  \chi^2_{2\|\bnu\|_1^2/\|\bnu\|_2^2}\,,\qquad\text{as}\qquad \sigma_*\to 0.
$$
This type of results are often referred to as Wilks' phenomenon.
\end{remark}

The proof of Theorem~\ref{th:1} is rather technical and, therefore, is deferred to the Appendix. Let us simply mention here
that we present the proof in a slightly more general case $\bc\in\calF_{s,L}$ with a smoothness $s\in(0,1]$. It appears
from the proof that the convergence stated in (\ref{conv}) holds if $s>7/8$, $N_{\sigma_*}\to\infty$ and
$\sigma_*^2 N_{\sigma_*}^{-2s+9/2}\log(N_{\sigma_*})\to 0$, as $\sigma_*\to 0$. We do not know whether the last condition on
$N_{\sigma_*}$ can be avoided by using other techniques, but it seems that in our proof there is no room for improvement
in order to relax this assumption. At a heuristic level, it is quite natural to avoid choosing $N_{\sigma_*}$ too large. Indeed,
large $N_{\sigma_*}$ leads to undersmoothing in the problem of estimating the quadratic functional $\|\bc-\ee(\tau)\circ\bc^\diese\|_2^2$.
Therefore, the test statistic corresponds to registration of two curves in which the signal is dominated by noise, which is clearly
not a favorable situation for performing curve registration.

\begin{remark}
The $p$-value of the aforementioned test based on the Gaussian or chi-squared approximation can be used as a measure
of the goodness-of-fit or, in other terms, as a measure of alignment for the pair of curves under consideration. If
the observed two noisy curves lead to the data $\by^{\bullet,\diese}$, then the (asymptotic) $p$-value is defined as
$$
\alpha^*=\Phi\Big(\frac{\Delta_{\bnu,\sigma_*}(\by^{\bullet,\diese})-\|\bnu\|_1}{\|\bnu\|_2}\Bigg),
$$
where $\Phi$ stands for the c.d.f.\ of the standard Gaussian distribution.
\end{remark}

Note that a similar result in the model of regression has been established by \citep{Hardle90}[Theorem 3]. Their results are
based on another test statistics, involving kernel smoothing, and hold under more stringent assumptions, like the existence of
a uniformly continuous second-order derivative of the regression function. Remark also that under these assumptions, the test
presented in  \citep{Hardle90} is definitely not rate optimal in the sense of minimax rate of separation \citep{Ingster_Suslina},
while our test is minimax rate optimal up to logarithmic factors, as proved in the companion paper~\citep{Collier2012}. Finally,
\citep{Hardle90} does not provide any result on the power of the proposed test statistics.

So far, we have only focused on the behavior of the test under the null without paying attention on what happens
under the alternative. The next theorem fills this gap by establishing the consistency of the test defined by the critical
region (\ref{eq:13}).

\begin{theorem}\label{th:2}
Let condition \textbf{\bf (C)} be satisfied and let $\sigma_*^4N_{\sigma_*}$ tend to $0$ as $\sigma_*\to 0$.
Then the test statistic  $T_{\sigma_*}=\frac{\Delta_{\bnu,\sigma_*}(\bY^{\bullet,\diese})-\|\bnu\|_1}{\|\bnu\|_2}$ diverges under  $\mathbf{H_1}$, \textit{i.e.},
$T_{\sigma_*} \convproba +\infty$ as $\sigma_*\to 0$.
\end{theorem}

In other words, Theorem~\ref{th:2} establishes the convergence to one of the power of the test defined via (\ref{eq:13})
as the noise level $\sigma_*$ tends to $0$.

The previous theorem tells us nothing about the (minimax) rate of separation of the null hypothesis from the alternative.
In other words, Theorem~\ref{th:2} does not provide the rate of divergence of $T_{\sigma_*}$. Under more stringent assumptions on
the weight sequence $\nnu$ and when $\sigma=\sigma^\diese$, the minimax approach is developed in the companion paper~\cite{Collier2012}.
Here, we focus our attention on studying other aspects of the previously introduced methodology, more related to their applicability
in the area of computer vision.

\begin{remark}
It might be compelling to start the shifted-curve test by performing a test of equality of norms. Indeed, if there is enough evidence for
rejecting the hypothesis $\|\bc\|_2 = \|\bc^\diese\|_2$, then the hypothesis $c_j = e^{-\ii j\bar\tau^*}c_j^\diese$, $j\ge 1$, will   
be rejected as well. To perform the test of the equality of norms, one can use the procedure proposed in \cite{Comminges2013}, which is
proved to achieve optimal rates of separation.
\end{remark}

\section{Some extensions}\label{sec:ext}

This section presents two possible extensions of the methodology developed in foregoing sections. They stem from practical considerations and
concern the case of multidimensional curves and the setting of unknown noise magnitude.

\subsection{Multidimensional signals}

Theorems~\ref{th:1} and \ref{th:2} can be straightforwardly extended to the case of multidimensional curves $f$ and $f^\diese$
from $\RR\to\RR^d$, with an arbitrary integer $d\ge 1$. More precisely, we may assume that the observations
$\{\bX_i,\ldots,\bX_n;\bX^\diese_1,\ldots,\bX^\diese_{n^\diese}\}$
are $\RR^d$ values random processes given by
$$
\bX_i(t)=\int_0^t \boldf(u)\,du+\diag(\bS)\bB_i(t),\quad \bX_\ell^\diese(t)=\int_0^t \boldf^\diese(u)\,du+\diag(\bS^\diese) \bB_\ell^\diese(t),
$$
where $\diag(\bS)$ and $\diag(\bS^\diese)$ are diagonal $d\times d$ matrices with positive entries and  the random processes
$\{\bB_1,\ldots,\bB_n,\bB_1^\diese,\ldots,\bB^\diese_{n^\diese}\}$ are independent $d$-dimensional Brownian motions. In order to test the null hypothesis
\begin{equation}\label{eq:3b}
\mathbf{H_0}:\exists\; (\tau^*,{\mathbf a^*})\in [0,1]\times\RR^d\qquad \text{ such that\ }\qquad
\boldf(\cdot)=\boldf^\diese(\cdot+\tau^*)+{\mathbf a^*},
\end{equation}
it suffices to compute the Fourier coefficients
\begin{align*}
\big[\bY_j,\bY_j^\diese\big]=\int_0^1 e^{2\pi \ii jt}\,d\bigg[\frac1{n}\sum_{i=1}^n \bX_i(t), \frac1{n^\diese}\sum_{\ell=1}^{n^\diese} \bX_\ell^\diese(t)\bigg]
\qquad j= 1,2,\ldots
\end{align*}
and to evaluate the test statistic
$$
\Delta(\bY^{\bullet,\diese})=\frac{1}{2}\min_{\tau\in[0,2\pi]}\sum_{j\in\NN} \nu_j
\big\|(\diag(\ssigma)^2+\diag(\ssigma^\diese)^2)^{-1/2}\big(\bY_j-e^{-\ii j\tau}\bY_j^\diese\big)\big\|_2^2\,,
$$
where $\ssigma = \bS/\sqrt{2n}$ and $\ssigma^\diese = \bS^\diese/\sqrt{2n^\diese}$. One can easily check that if  $\mathbf H_0$ is true, then
under the assumptions of Theorem~\ref{th:1}, as $\sigma_* = \max(\|\ssigma\|_\infty,\|\ssigma^\diese\|_\infty)\to 0$, the random variable
$T_{\sigma_*}=(\Delta(\bY^{\bullet,\diese}) - d\|\nnu\|_1)/(\sqrt{d}\|\nnu\|_2)$ converges in distribution to a standard Gaussian random variable.
Furthermore, under $\mathbf H_1$, we have $T_{\sigma_*}\to\infty$ in probability provided that the assumptions of Theorem~\ref{th:2} are fulfilled.

\subsection{Unknown noise level}

In several application it is not realistic to assume that the magnitude of noise, denoted by $(\sigma,\sigma^\diese)$ is known in advance.
In such a situation, the testing procedure defined by the critical region (\ref{eq:13}) cannot be applied, since the test statistic
$\Delta_{\bnu,\sigma_*}(\bY^{\bullet,\diese})$ depends on $\sigma^2+\sigma^\diese{}^2$. We describe below one possible approach to address this issue.
Note that in order to make this setting meaningful, we assume that the noisy Fourier coefficients $Y_j$ and $Y^\diese_j$ are observable
only for a finite number of indices $j\in\{1,\ldots,p\}$. Therefore, we assume in this section that $\bY$ and $\bY^\diese$ are complex valued vectors of dimension $p$.

Adopting the same strategy as before, we aim at defining a testing procedure based on the principle of penalized likelihood ratio evaluation. However, when
the pair $(\sigma,\sigma^\diese)$ is unknown, the expression (\ref{eq:8}) for the negative log-likelihood is not valid anymore.  Instead, up to some irrelevant
summands, we have
\begin{equation}\label{eq:8'}
\ell(\bY^{\bullet,\diese},\bc^{\bullet,\diese},\sigma^{\bullet,\diese})= p(\log\sigma + \log\sigma^\diese)+ \frac{\|\bY-\bc\|_2^2}{2\sigma^2}+\frac{\|\bY^\diese-\bc^\diese\|_2^2}{2\sigma^\diese{}^2}.
\end{equation}
Therefore, given a vector of weights $\oomega\in \RR_+^p$, the penalized log-likelihood is defined as
\begin{align}
p\ell(\bY^{\bullet,\diese},\bc^{\bullet,\diese},\sigma^{\bullet,\diese})=&
\ell(\bY^{\bullet,\diese},\bc^{\bullet,\diese},\sigma^{\bullet,\diese})+
\frac{\|\bc\|_{\oomega,2}^2}{2\sigma^2}+
\frac{\|\bc^\diese\|_{\oomega,2}^2}{2\sigma^\diese{}^2}\:.\label{eq:9'}
\end{align}
Thus, the test statistic to be used is the difference between the minimum of the penalized log-likelihood constrained to $\mathbf{H_0}$
and the unconstrained minimum of the penalized log-likelihood, that is
\begin{align}
\Delta(\bY^{\bullet,\diese})=\min_{\sigma^{\bullet,\diese}}\min_{\bc^{\bullet,\diese}:\mathbf{H_0}
\text{ is true}}p\ell(\bY^{\bullet,\diese},\bc^{\bullet,\diese},\sigma^{\bullet,\diese})-\min_{\sigma^{\bullet,\diese}}
\min_{\bc^{\bullet,\diese}}p\ell(\bY^{\bullet,\diese},\bc^{\bullet,\diese},\sigma^{\bullet,\diese}).\label{eq:10'}
\end{align}
Denoting by $\nnu$ the vector $ (1/(1+\omega_1),\ldots,1/(1+\omega_p))^\top$ and by $1-\nnu$ the vector $(1-\nu_1,\ldots,1-\nu_p)^\top$,
and restricting the minimization to $\sigma=\sigma^\diese$, we get
$$
\Delta(\bY^{\bullet,\diese})=
p\log\bigg(1+\frac{\min_{\tau\in[0,2\pi[} \|\bY-\ee(\tau)\circ\bY^\diese\|_{2,\nnu}^2}{2(\|\bY\|_{2,1-\nnu}^2+\|\bY^\diese\|_{2,1-\nnu}^2)}\bigg).
$$
Let $\alpha\in(0,1)$ be a prescribed significance level and let us introduce the statistic
\begin{align}\label{eq:11'}
\tilde \Delta(\bY^{\bullet,\diese})  =
\frac{\|1-\nnu\|_1}{\|\bY\|_{2,1-\nnu}^2+\|\bY^\diese\|_{2,1-\nnu}^2}
{\min_{\tau\in[0,2\pi[} \|\bY-\ee(\tau)\circ\bY^\diese\|_{2,\nnu}^2}.
\end{align}
Intuitively, this new test statistic $\tilde \Delta(\bY^{\bullet,\diese})$  can be seen as an estimator
of $\frac1{2(\sigma^2+\sigma^\diese{}^2)} {\min_{\tau} \|\bY-\ee(\tau)\circ\bY^\diese\|_{2,\nnu}^2}$
used in the setting of known noise level. Therefore, it is not so much a surprise that the critical region we deduce from (\ref{eq:11'})
is of the form
$$
\tilde \Delta(\bY^{\bullet,\diese}) \ge C_{\nnu,\alpha},
$$
where $C_{\nnu,\alpha}$ is a given threshold. To propose a choice of this threshold that leads to a test of  asymptotic level $\alpha$,
the asymptotic distribution of $\tilde \Delta(\bY^{\bullet,\diese}) $ should be characterized under the null hypothesis. Let us introduce the statistic
\begin{align}\label{test:t}
T(\bY^{\bullet,\diese})=\frac{\tilde \Delta(\bY^{\bullet,\diese}) - \|\nnu\|_1}{\|\nnu\|_2}.
\end{align}
In order to simplify the presentation, we develop the subsequent arguments only in the case $\sigma = \sigmadiese=\sigma_*$.
We further assume that the noisy Fourier coefficients $\bY$ and $\bYdiese$ are generated by a process described in the Introduction,
cf.~(\ref{eq:5}), which means that $\sigma_*$ is equal to $s/\sqrt{2n}$ with a known parameter $n$ and an unknown factor $s$. The asymptotic
setting $\sigma_*\to 0$ corresponds then to the standard ``large sample asymptotics'' $n\to\infty$. The main advantage of this
setting is that it allows us to use the knowledge of $n$ in the choice of the sequence $\bnu$.

\begin{theorem}\label{th:3}
Let  $\bc\in\calF_{1,L}$ and let the sequence $\bnu$ satisfy conditions {\bf(A)} and {\bf(B)} with
an integer $N_{\sigma_*}$ that depends only on $n$ so that $N_{\sigma_*}\to+\infty$ and $\sigma_*N_{\sigma_*}=O(1)$
when $\sigma_*$ tends to zero.
Assume, in addition, that $p\ge 2N_{\sigma_*}$ is large enough to satisfy  $(p-N_{\sigma_*})\sigma_*^2
N_{\sigma_*}^{3/2}\to\infty$  and that
for some constant $c'>0$,  $\max_{j\ge 1} j^{-2}(1-\nu_j) \le c' N_{\sigma_*}^{-2}$.
Then, under the null hypothesis, if $\sigma=\sigma^\diese$ tends to zero, the test statistic $T(\bY^{\bullet,\diese})$
satisfies
\begin{align*}
T(\bY^{\bullet,\diese})&\le \sum_{j=1}^{p} \frac{\nu_j(|\xi_j|^2-2)}{2\|\nnu\|_2}+
\frac{\|\nnu\|_1}{4\|\nnu\|_2}\sum_{j=1}^p\frac{1-\nu_j}{\|1-\nnu\|_1}\;\big(4-|\xi_j|^2-|\xi_j^\diese|^2\big)
+\frac{O_P(1)}{(p-N_{\sigma_*})\sigma_*^2N_{\sigma_*}^{3/2}\wedge (p-N_{\sigma_*})^{1/2}},
\end{align*}
where $\{\xi_j;\xi_j^\diese\}$ are i.i.d.\  $\mathcal N_\CC(0,1)$ random variables.
\end{theorem}

The proof of this theorem is postponed to the Appendix. Instead, we discuss here some relevant consequences of it.
First of all, note that a simple application of Lyapunov's central limit theorem implies that under the conditions
$N_{\sigma_*}\to\infty$ and $(p-N_{\sigma_*})^{2}/p \to\infty$ both
sums  $\sum_{j=1}^{p} \frac{\nu_j}{2\|\nnu\|_2}(|\xi_j|^2-2)$ and $\sum_{j=1}^p\frac{1-\nu_j}{2\sqrt{2}\|1-\nnu\|_2}\;\big(4-|\xi_j|^2-|\xi_j^\diese|^2\big)$
converge in distribution to the standard Gaussian distribution. Therefore, the test defined by the critical region
\begin{align}\label{test:3}
T(\bY^{\bullet,\diese}) \ge \min_{0<\beta<1}\bigg(
z_{1-\beta}+\frac{\|\nnu\|_1\|1-\nnu\|_2}{\sqrt{2}\|\nnu\|_2\|1-\nnu\|_1}z_{1-\alpha+\beta}\bigg)
\end{align}
is asymptotically of level not larger than $\alpha$.
Note also that a more precise critical region can be deduced from Theorem~\ref{th:3} without relying on the central limit theorem. Indeed,
the main terms $\sum_{j=1}^{p} \frac{\nu_j}{2\|\nnu\|_2}(|\xi_j|^2-2)$ and $\sum_{j=1}^p\frac{1-\nu_j}{2\sqrt{2}\|1-\nnu\|_2}\;\big(4-|\xi_j|^2-|\xi_j^\diese|^2\big)$ have parameter free distributions, the quantiles
of which can be determined numerically by means of Monte Carlo simulations. In the case of projection weights, one can also use the quantiles of
chi squared distributions.

To conclude this section, let us have a closer look at the assumptions of the last theorem. The conditions {\bf(A)}, {\bf(B)}
and $\max_{j\not = 0} j^{-2}(1-\nu_j) \le c' N_{\sigma_*}^{-2}$ are satisfied for most weights used in practice. Thus, the most important conditions are
$(p-N_{\sigma_*})\sigma_*^2N_{\sigma_*}^{3/2}\to\infty$ and $\sigma_*N_{\sigma_*}\to 0$. The first of these two conditions ensures
that the error term coming from the estimation of the unknown noise level is small. The second one is a weak version of the condition
$\sigma_*N_{\sigma_*}^{5/2}\log N_{\sigma_*} = o(1)$ present in Theorem~\ref{th:1}, which ensures that we do not use a strongly
undersmoothed test statistic. We manage here to obtain a condition on $N_{\sigma_*}$ which is weaker than the corresponding condition
in Theorem~\ref{th:1} because we do not establish the asymptotic distribution of the test statistic but just an upper bound of the latter.
The choice of $p$ and $N_{\sigma_*}$ is particularly important for obtaining a test with a power close to one, especially in the case of
alternatives that are close to the null. However, the investigation of this point is out of scope of the present work. Let us just mention
that if we choose $N_{\sigma_*} = \sigma_*^{-\beta}$ and $p = 2\sigma_*^{-\gamma}$, the conditions of Theorem~\ref{th:3} are satisfied
if $\beta\le\min(\gamma,1)$ and $2\gamma+3\beta\ge 4$. Of course, these conditions are closely related to the assumption that the
unknown signal belongs to the smoothness class of regularity $1$.

\section{Numerical experiments}\label{sec:3}

We have implemented the proposed testing procedures (\ref{eq:13}) and (\ref{test:3}) in Matlab and carried out a certain number of
numerical experiments on synthetic data. The aim of these experiments is merely to show that the methodology developed in the present
paper is applicable and to give an illustration of how the different characteristics of the testing procedure, such as the significance
level and the power, depend on the noise variance $\sigma_*^2$ and on the shrinkage weights $\bnu$. We also aimed at comparing the
performance of the testing procedure (\ref{test:3}) with that of (\ref{eq:13}). In order to ensure the reproducibility, the Matlab code 
of the experiments reported in this section is made available on
\url{https://code.google.com/p/shifted-curve-testing/}.

\subsection{Behavior of the Type I error rate}\label{sec4.1}

In order to illustrate the convergence of the test statistic of the procedure (\ref{eq:13}) when $\sigma_*$ tends to zero and to assess the
Type I error rate, we conducted the following experiment. We chose as function $f$ the smoothed version of the HeaviSine function,
considered as a benchmark in the signal processing community, and computed its complex Fourier coefficients
$\{c_j;j=0,\ldots,10^6\}$. More precisely, the $j$th Fourier coefficients $c_j$ of $f$ is obtained by dividing by $j$ the
corresponding Fourier coefficient of the HeaviSine function.
\begin{center}
\begin{figure}[ht!]
\centerline{\includegraphics[width=0.85\textwidth]{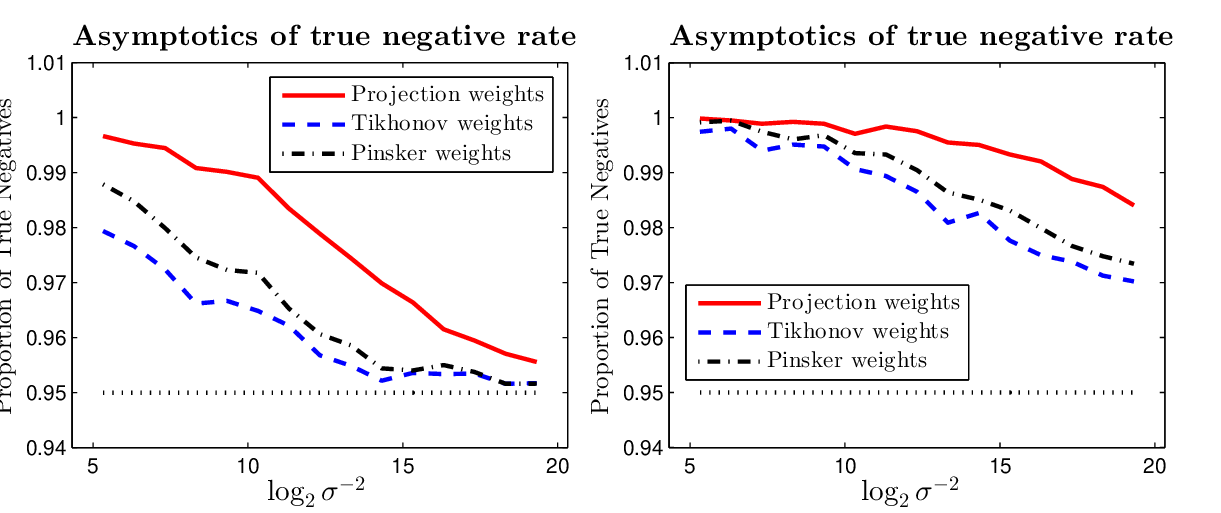}}
\caption{The proportion of true negatives in the experiment described in Section~\ref{sec4.1} as a function of $n =\log_2 \sigma^{-2}$
for three different shrinkage weights: projection (dashed line), Tikhonov (solid line) and Pinsker (dash-dotted line).
One can observe that for all the weights the proportion of true negatives converges to the nominal level $0.95$.
The left panel plots the results for the test procedure using the true noise level, while the right panel
plots the results for the procedure corresponding to unknown noise level.}
\label{fig:1}
\end{figure}
\end{center}

For each value of $n$ taken from the set $\{n_k=20\times 2^{k},\, k=1,\ldots,15\}$, we repeated $10^5$
times the following computations:
\vspace{-7pt}
\begin{itemize}\itemsep=0pt
\item[$\bullet$] set $\sigma_* = n^{-1/2}$ and\footnote{This value of $N_{\sigma_*}$ satisfies the
assumptions required by our theoretical results.} $N_{\sigma_*}=[50\sigma_*^{-1/2}]$,
\item[$\bullet$] generate the noisy sequence $\{Y_j; j=1,\ldots,N_{\sigma_*}\}$ by adding to $\{c_j\}$ an
i.i.d.\ $\mathcal N_{\mathbb C}(0,\sigma_*^2)$ sequence $\{\xi_j\}$,
\item[$\bullet$] randomly choose a parameter $\tau^*$ uniformly distributed in $[0,2\pi]$, independent of $\{\xi_j\}$,
\item[$\bullet$] generate the shifted noisy sequence $\{Y_j^\diese; j=1,\ldots,N_{\sigma_*}\}$ by adding to $\{e^{\ii j\tau^*}c_j\}$  an i.i.d.\
$\mathcal N_{\mathbb C}(0,\sigma_*^2)$ sequence  $\{\xi_j^\diese\}$, independent of $\{\xi_j\}$ and of $\tau^*$,
\item[$\bullet$] compute the three values of the test statistic $\Delta_{\bnu,\sigma_*}$ corresponding to the classical shrinkage weights
defined by (\ref{weights}) and compare these values with the threshold for $\alpha=5\%$.
\end{itemize}
We denote by $p_{\text{accept}}^{\text{proj}}(\sigma_*)$, $p_{\text{accept}}^{\text{Tikh}}(\sigma_*)$ and $p_{\text{accept}}^{\text{Pinsk}}(\sigma_*)$
the proportion of experiments (among $10^5$ that were realized) for which the value of the corresponding test statistic was lower than
the threshold, \textit{i.e.}, the proportion of experiments leading to the non-rejection of the null hypothesis.
We plotted in the left panel of Fig.~\ref{fig:1} the (linearly interpolated) curves $k\mapsto p_{\text{accept}}^{\text{proj}}(n_k)$,
$k\mapsto p_{\text{accept}}^{\text{Tikh}}(n_k)$ and
$k\mapsto p_{\text{accept}}^{\text{Pinsk}}(n_k)$. For the Pinsker weights, we (somewhat arbitrarily) chose
$\mu=2$, while the parameters of the Tikhonov weights were chosen as follows: $\kappa=1/2$ and $\mu=2$.
The graphs plotted in the left panel of Figure~\ref{fig:1} show that the proportion of true negatives is very close to the nominal
level $0.95$, irrespectively from the choice of the weight sequence. Furthermore, when $\sigma_*$ goes to zero (that is when $n$
becomes large) the empirical Type I error rate converges to the nominal level. This is in line with the claim of Theorem~\ref{th:1}.

We also carried out the same experiment for the test statistic (\ref{eq:11'}) that does not require the knowledge of the noise levels
$\sigma$ and $\sigmadiese$. The data generation process was the same as before, except that $\sigma_*$ was defined as $s/\sqrt{n}$,
where $s$ was drawn at random uniformly in the interval $[1,4]$. The cut-off parameter $N_{\sigma_*} = N_n$ was then set to
$[50n^{1/4}]$, independently of the value of $s$. The number $p$ of Fourier coefficients was chosen proportional to $N^{3/2}$.
It is interesting to observe that in this situation as well the observed proportion of true
negatives is dominating  the nominal level. This test is more conservative than the one based on the full knowledge of the noise
level and this is not surprising since the threshold used in (\ref{test:3}) is not based on as asymptotic distribution of the test statistic but merely
on an asymptotic upper bound. It should also be noted the noise-level-adaptive procedure has a computational complexity which is slightly higher than the one of the test procedure for known $\sigma_*$. This increase of computational complexity is due to the fact that the estimation of $\sigma_*^2$ requires computing the Fourier
coefficients of the observed signals corresponding to high frequencies.

\subsection{Power of the tests}\label{sec:A3}

In the previous experiment, we illustrated the behavior of the penalized likelihood ratio test under the null hypothesis.
The aim of the second experiment is to show what happens under the alternative. To this end,  we still use the smoothed HeaviSine
function as signal $f$ and define $f^\diese=f+\gamma \varphi$, where $\gamma$ is a real parameter. Two cases are considered:
$\varphi(t)=c\cos(4t)$ and $\varphi(t)=c/(1+t^2)$, where $c$ is a constant ensuring that $\varphi$ has an $L^2$ norm equal to that
of $f$. For the sake of the conciseness, only the results obtained for the projection weights are reported.

In the experiment described in this section, we chose $n$ from $\{ n(k) = 2^k:k=1,\ldots,4\}$ and $\gamma$ from
$\{\gamma(\ell) = 0.1\ell:\ell= 1,\ldots,15\}$.  For each value of the pair $(n(k),\gamma(\ell))$, we repeated $5000$ times
the following computations:
\vspace{-7pt}
\begin{itemize}\itemsep=0pt
\item[$\bullet$] set $\sigma_*=n(k)^{-1/2}$ and $N_{\sigma_*}=[50\sigma_*^{-1/2}]$,
\item[$\bullet$] compute the complex Fourier coefficients
$\{c_j;j=1,\ldots,10^6\}$ and $\{c_j^\diese;j=1,\ldots,10^6\}$ of $f$ and $f^\diese$, respectively,
\item[$\bullet$] generate the noisy sequence $\{Y_j; j=1,\ldots,N_{\sigma_*}\}$ by adding to $\{c_j\}$ an i.i.d.\ $\mathcal N_{\mathbb C}(0,\sigma_*^2)$
sequence $\{\xi_j\}$,
\item[$\bullet$] generate the shifted noisy sequence $\{Y_j^\diese; j=1,\ldots,N_{\sigma_*}\}$ by adding to $\{c_j^\diese\}$  an i.i.d.\
$\mathcal N_{\mathbb C}(0,\sigma_*^2)$ sequence  $\{\xi_j^\diese\}$, independent of $\{\xi_j\}$,
\item[$\bullet$] compute the value of the test statistic $\Delta_{\bnu,\sigma_*}$ corresponding to the projection weights and compare this value
with the threshold for $\alpha=5\%$.
\end{itemize}
To demonstrate the behavior of the test under $\mathbf{H_1}$ when the distance between the null and the alternative varies,
we computed the proportion of true positives, also called the empirical power, among the $5000$ simulated random samples.
The results, plotted in Fig.\ ~\ref{fig:2} show that even for moderately small values of $\gamma$, the test succeeds in
taking the correct decision. We also clearly observe the convergence of the test since the curves corresponding to small
values of $\sigma_*$ (or, equivalently, large values of $n$) are at the left of the curves corresponding to larger values
of $\sigma_*$. Note also that the results for $\varphi(t) = c \cos(4t)$ and $\varphi(t) = c/(1+t^2)$ are quite comparable.

\begin{figure}[ht!]
\centerline{\includegraphics[width=0.85\textwidth]{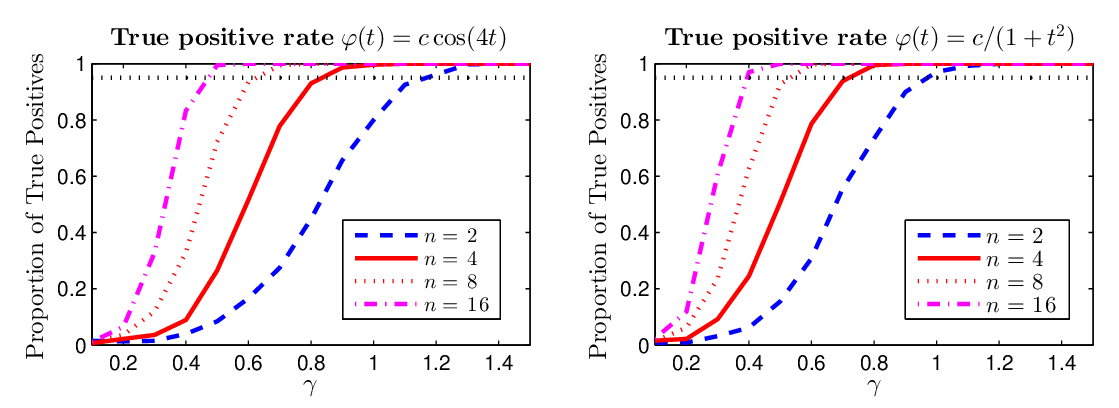}}
\caption{The proportion of true positives in the experiment described in Section~\ref{sec:A3} as a function of the parameter
$\gamma$ measuring the ``distance'' from the null. The parameter $n$ is defined as $\sigma_*^{-2}$, where $\sigma_*$ is the noise
level. We can observe that even for relatively small values of $\gamma$ the test has a power close to one.}
\label{fig:2}
\end{figure}

To further investigate the power of the generalized likelihood ratio test described in previous sections, we
carried out the last experiment with a nonsmooth function $\varphi$. More precisely, we chose as $\varphi$ the
unsmoothed version of the function $\psi(t) = 1/(1+t^2)$, in the sense that the Fourier coefficient $c_j$ of $\varphi$
is equal to the corresponding Fourier coefficient of $\psi(\cdot)$ multiplied by $j$, for $j \ge 1$. The function $\varphi$
obtained in this way is then normalized to have a $L^2$-norm equal to that of $f$. Note that the nonsmoothness of $\varphi$
implies that of $f^{\diese}$. The results are plotted in Fig.~\ref{fig:4:4}. When compared with Fig.~\ref{fig:2}, this
plot clearly shows that the power of the test gets deteriorated in the nonsmooth case. Indeed, in order to achieve a behavior
of the power similar to that of Fig.~\ref{fig:2}, we need to multiply $\gamma$ by a factor close to 10.

\begin{figure}[ht!]
\centerline{\includegraphics[width=0.45\textwidth]{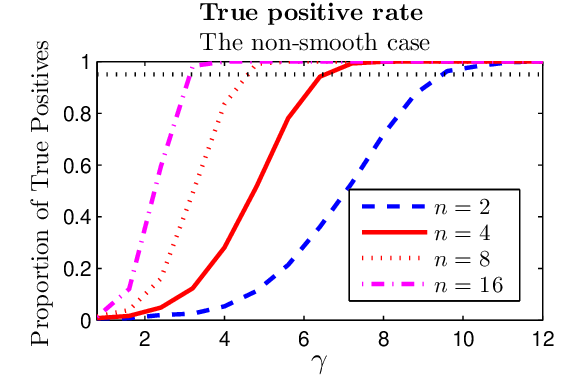}}
\caption{The proportion of true positives in the experiment described in Section~\ref{sec:A3} as a function of the parameter
$\gamma$ measuring the ``distance'' from the null. The ``perturbation'' function $\varphi$ is the unsmoothed version of the function
$c/(1+t^2)$. Comparing this plot to Fig.~\ref{fig:2}, we see that, for a given $n$, the values of $\gamma$ leading to a power
close to one are much larger. }
\label{fig:4:4}
\end{figure}

Finally, we performed the same experiment for the noise-level-adaptive procedure defined by
(\ref{eq:11'}), (\ref{test:t}) and (\ref{test:3}). The differences compared to the protocol described
in the beginning of this subsection were that the values of $n$ were chosen among $\{n(k) = 10\times 2^k: k = 1,\ldots,4\}$
and $\sigma_*$ and $N_{\sigma_*}$ were set to $sn^{-1/2}$ and $[50n^{1/4}]$, respectively.
As in the nonsmooth case, here also we observe that the rate of convergence gets deteriorated. The noise-level-adaptive
procedure has a power equivalent to that of the original procedure for noise levels that are divided by $\sqrt{10}$.
This is in part explained by the fact that the adaptive test is conservative (cf.~the right panel of Fig.~\ref{fig:1}) but
also by the fact that the substitution of the noise level by an estimator results in an increased stochastic error.

\begin{figure}[ht!]
\centerline{\includegraphics[width=0.85\textwidth]{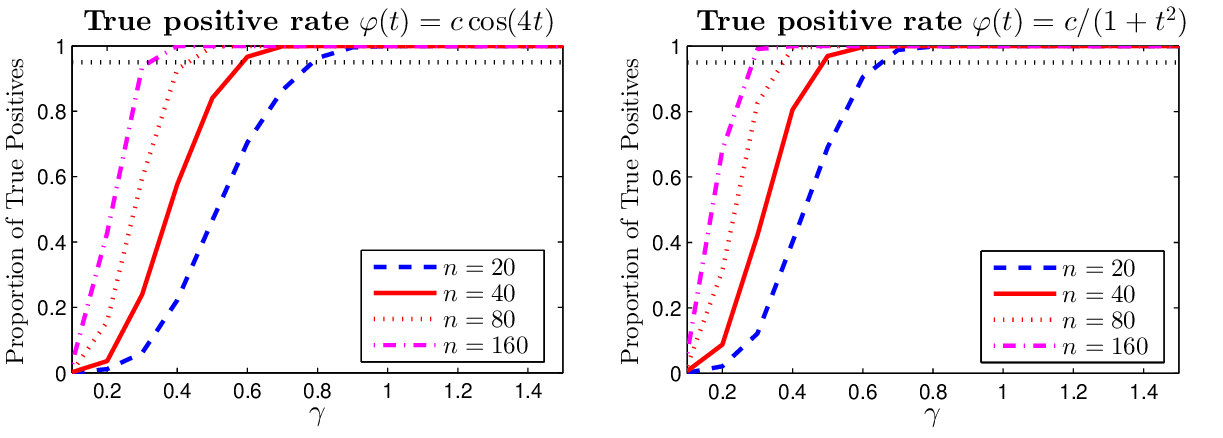}}
\caption{The proportion of true positives in the experiment described in Section~\ref{sec:A3} as a function of the parameter
$\gamma$ measuring the ``distance'' from the null. As opposed to Fig.~\ref{fig:2}, here we applied the noise-level-adaptive test
defined by  (\ref{eq:11'}), (\ref{test:t}) and (\ref{test:3}). The parameter $n$ is defined as $\sigma_*^{-2}$, where $\sigma_*$ is the noise
level. We can observe that in order to get plots similar to those of Fig.~\ref{fig:2}, we used much smaller values of $\sigma_*$.
This means that the convergence of the noise-level adaptive-test is slower than that of the original generalized likelihood
ratio test.}
\label{fig:4:3}
\end{figure}

\section{Application to keypoint matching}\label{sec:LoFT}

As mentioned in the introduction, the methodology developed in previous sections may be applied to the problem of keypoint matching
in computer vision. More specifically, for a pair of digital images $I$ and $I^\diese$ representing the same 3D object, the task of keypoint matching
consists in finding pairs of image points $(\mathbf x_0, \mathbf x_0^\diese)$ in images $I$ and $I^\diese$ respectively, corresponding to the same 3D point.
For more details on this and related topics, we refer the interested reader to the book \cite{Hartley_Ziss}. For our purposes here, we assume that
we are given a pair of points  $(\mathbf x_0, \mathbf x_0^\diese)$ in images $I$ and $I^\diese$ respectively and the goal is to decide whether they are the projections
of the same 3D point or not. In fact, we will tackle a slightly simpler\footnote{Note here that using the state-of-the-art techniques
of keypoint detection based on the differences of Gaussians, see \cite{Lowe}, one can recover a rather reliable value of the scale parameter
for every keypoint. Using this scale parameter, the problem of testing for a similarity transform reduces to the problem of testing for a rotation
that we consider in this section.} problem corresponding to deciding whether a neighborhood of $\mathbf x_0$ in the image $I$ coincides with a
neighborhood of $\mathbf x_0^\diese$ in the image $I^\diese$ up to a rotation. Of course, this problem is made harder by the fact that the images
are contaminated by noise.

The plan of this section is as follows. As a first step, we present a new definition of a local descriptor (termed LoFT for Localized Fourier Transform)
of a keypoint $\mathbf x_0$ in some image $I$. This local descriptor is based on the Fourier coefficients of some mapping related to the local neighborhood
of the image $I$ around $\mathbf x_0$. Therefore, it is particularly well suited for testing for rotation between two keypoints. In a second step, we define
a matching criterion: a $\{0,1\}$-valued mapping that takes as input pairs $(\mathbf x_0, I)$ and $(\mathbf x_0^\diese, I^\diese)$  and outputs 1 if and only
if $\mathbf x_0$ and $\mathbf x_0^\diese$ are classified as matching points (\textit{i.e.}, corresponding to the same 3D point). Finally, as a third step, we
perform several experiments showing the potential of the proposed approach.

\begin{figure}[ht!]
\centerline{\includegraphics[height=3.6cm]{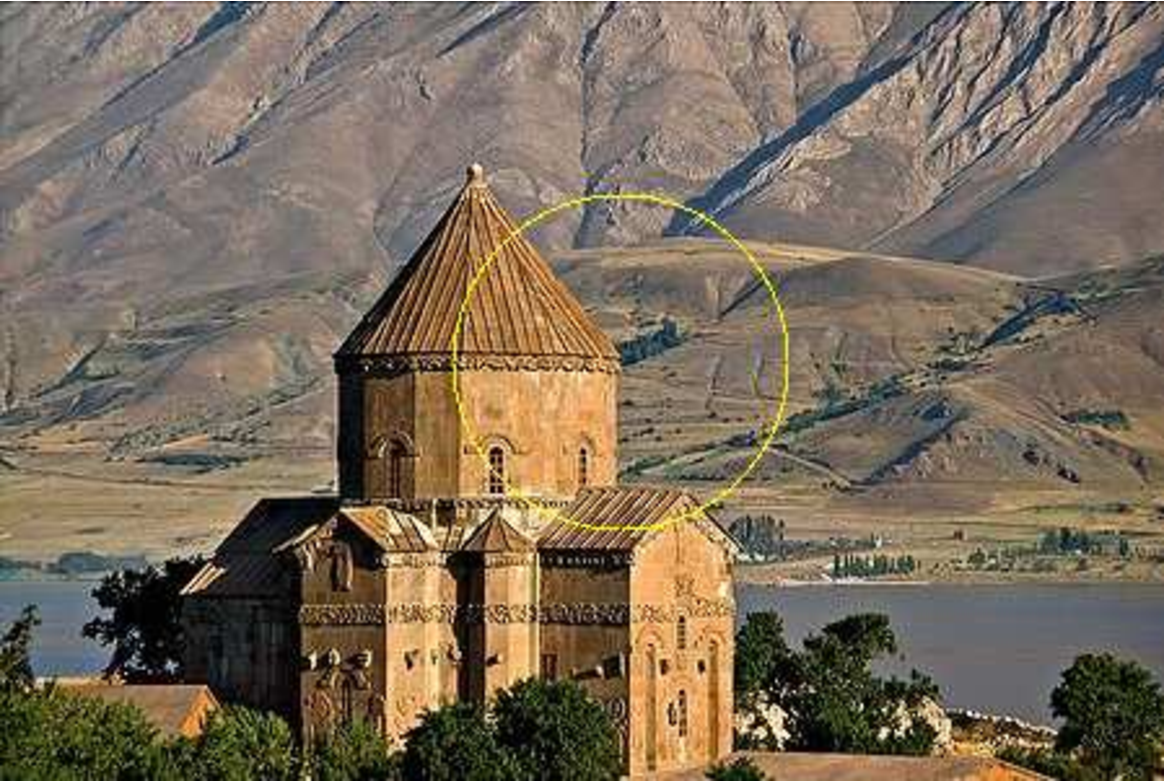}\hspace{15pt}
\includegraphics[height=3.6cm]{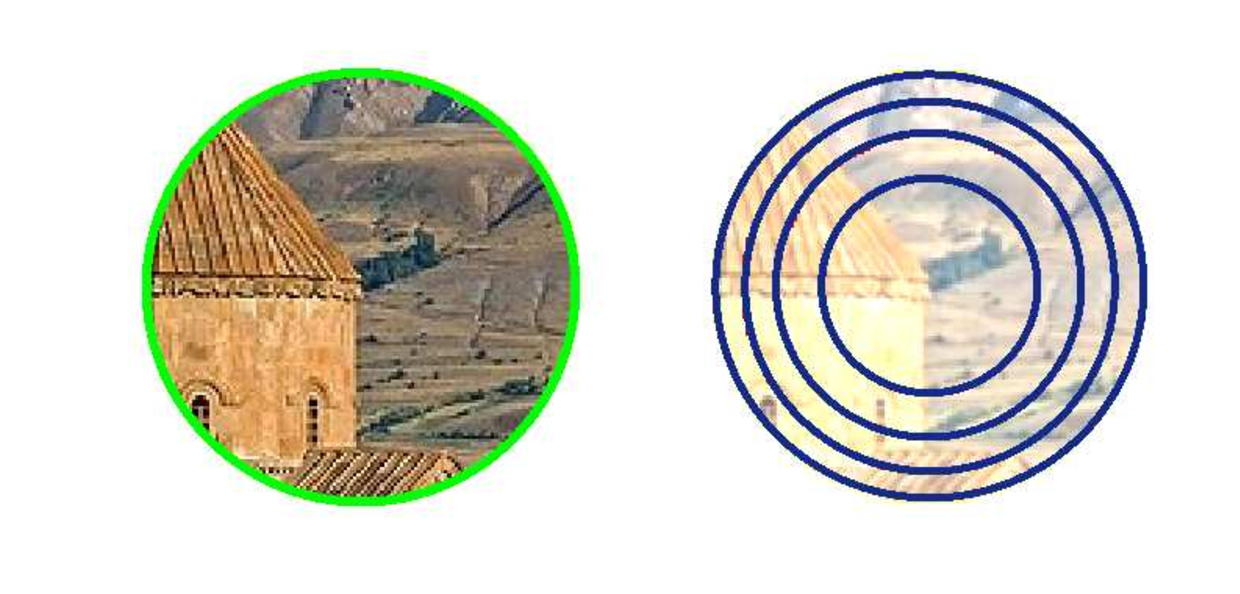}}
\hbox to \hsize{\bf\hspace{75pt}(a) \hspace{135pt} (b) \hspace{81pt}(c)\hfill}
\bigskip
\centerline{
\includegraphics[height=4.5cm]{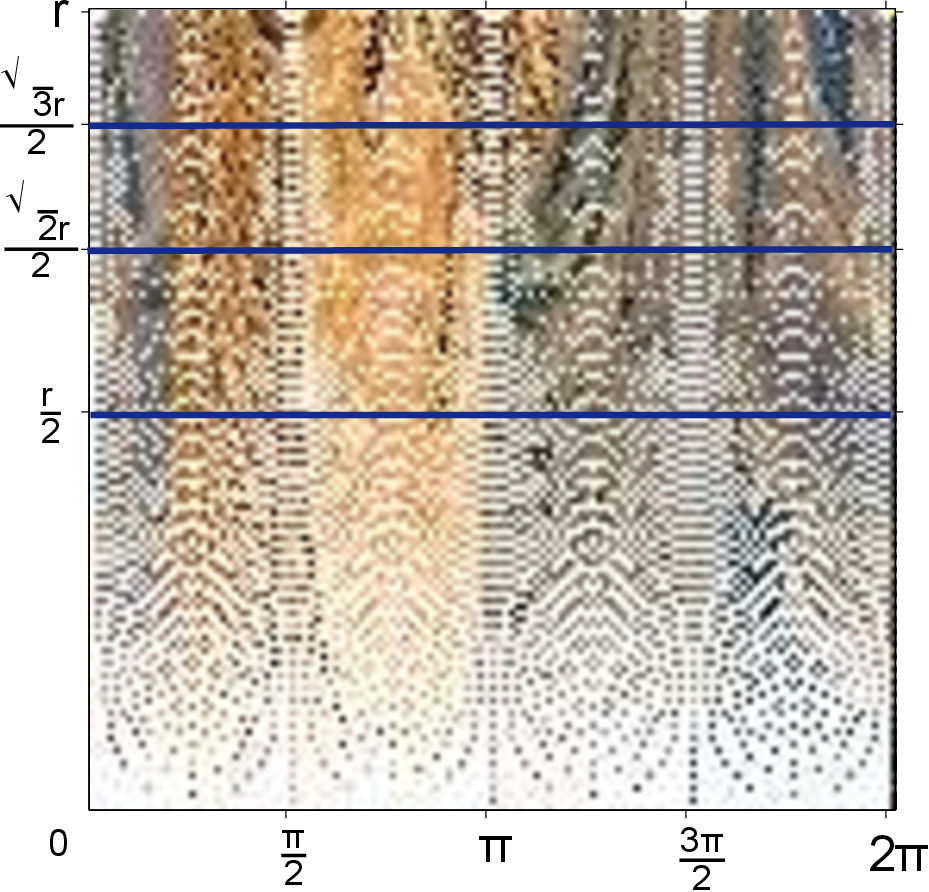}\hspace{15pt}
\includegraphics[height=4.5cm]{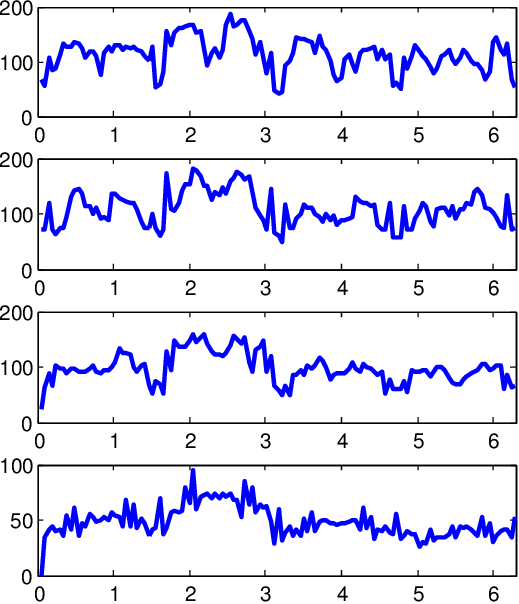}\hspace{15pt}
\includegraphics[height=4.5cm]{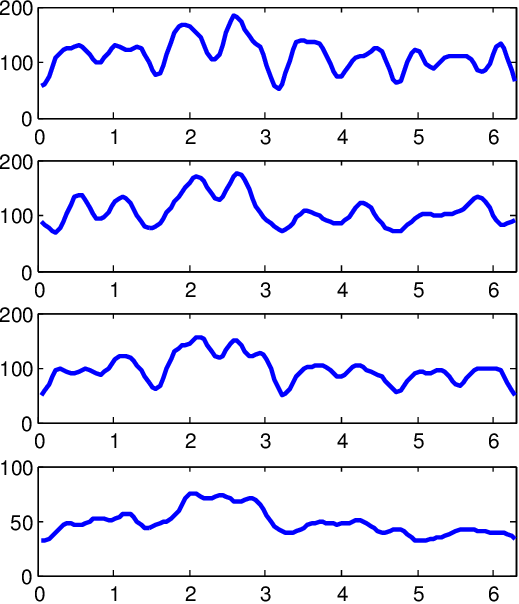}}
\hbox to \hsize{\bf\hspace{80pt}(d) \hspace{110pt} (e) \hspace{102pt}(f)\hfill}
\caption{Illustration of the construction of the LoFT descriptor of a keypoint $\mathbf x_0$ in an image $I$. Once a keypoint is chosen along with the radius
of the neighborhood to be considered (a,b), we split the corresponding subimage into 4 rings of equal areas represented in (c). In (d), we plot
the pixel intensities as a function of polar coordinates of the pixel. Each band corresponds to a ring in (c). The four curves in (e) are obtained by averaging
the pixel intensities corresponding to the same angle within each band of (d). Smoothed versions of these curves obtained by removing high frequencies are plotted in (f).}
\label{fig:3}
\end{figure}

\subsection{LoFT descriptor}

We start by describing the construction of the LoFT descriptor of a point $\mathbf x_0$ in a digital image $I$. For the purpose of illustration we
use color images, but all the experiments were conducted on grayscale images only. As any other construction of local descriptor, it is
necessary to choose a radius $r>0$ that specifies the neighborhood around $\mathbf x_0$. In all our experiments we used $r=32$ pixels, which seems to
lead to good results. Thus, we restrict the image $I$ to the disc $\mathcal D(\mathbf x_0,r)$ with center $\mathbf x_0$ and radius $r$, as shown in Fig.\ ~\ref{fig:4} (a) and (b). Note that this disc contains approximately $\pi r^2 > 3000$ pixels, but we will encode the restriction of $I$ to
$\mathcal D(\mathbf x_0,r)$ by a vector of size $128$.

The main idea consists in considering the function $\boldsymbol X:[0,2\pi]\to \RR^4$ defined by
\begin{equation}\label{eq:5.0}
\boldsymbol X(t) =
\begin{bmatrix}
X_1(t)\\
\vdots\\
X_4(t)
\end{bmatrix},\qquad X_\ell(t) = \int_{\sqrt{\ell-1}r/2}^{\sqrt{\ell}r/2} I\big(\mathbf x_0 + u[\sin(t), \cos(t)]\big)\,du,\quad \ell=1,\ldots,4.
\end{equation}
In other terms, each $X_\ell(t)$ describes the behavior of $I$ on some ring centered at $\mathbf x_0$, cf.\ Fig.\  \ref{fig:3}(c) for
an illustration. As shown in Fig.\  \ref{fig:3}(e), because of noise and textures present in the images, the functions $X_i$ are highly
nonsmooth. Since the details are not necessarily very informative when matching two image regions, we suggest to smooth out the functions
$X_\ell$ by removing high frequency Fourier coefficients, cf. Fig.\  \ref{fig:3}(f). The resulting descriptor is the vector composed of the
first $k$ Fourier coefficients
\begin{equation}\label{eq:5.2}
\bY\!_j = \frac1{\sqrt{2\pi}}\int_0^{2\pi} \bX(t) e^{\ii j t}\,dt,\qquad j=1,\ldots,k.
\end{equation}
To get a descriptor of size 128 (a complex number is encoded as two real numbers corresponding to its real and imaginary parts), we chose
$k=16$. The computation of each element of  the descriptor requires thus to evaluate an integral of the form
\begin{equation}\label{eq:5.3}
Y_j(\ell) = \frac1{\sqrt{2\pi}}\int_0^{2\pi} \int_{\sqrt{\ell-1}r/2}^{\sqrt{\ell}r/2} e^{\ii j t} I\big(\mathbf x_0 + u[\sin(t), \cos(t)]\big)\,du\,dt.
\end{equation}
Note that we dispose only a regularly sampled version of the image $I$ in the Cartesian coordinate system. This results in a nonregular sampling
in the polar coordinate system, illustrated in Fig.~\ref{fig:3}(d). The integrals are then approximated by the corresponding Riemann sums; the
rings being chosen so that they contain approximately the same number of sampled points, the qualities of these approximations are roughly equivalent.

\subsection{Matching criterion}

The rationale for using this function $\boldsymbol X$ is the following. Assume that $\mathbf x_0$ and $\mathbf x_0^\diese$ are
two true matches and that the images are observed without noise. That is to say that $\mathbf x_0$ and $\mathbf x_0^\diese$ are two points in
$I$ and $I^\diese$, respectively, such that if we rotate $I$ by some angle $\tau^*$ around $\mathbf x_0$ then in the neighborhood of
$\mathbf x_0$ of radius $r$ the image $I$ coincides with $I^\diese$ in the neighborhood of $\mathbf x'$. Or,  mathematically speaking,
$\forall (u,t)\in[0,r]\times [0,2\pi]$,
\begin{equation}\label{eq:5.1}
I\big(\mathbf x_0 + u[\sin(t-\tau^*),\cos(t-\tau^*)]\big) = I^\diese\big(\mathbf x_0^\diese+u[\sin t, \cos t]\big).
\end{equation}
Then, by simple integration and using (\ref{eq:5.0}) one checks that
\begin{equation}\label{eq:5.4}
\boldsymbol X(t-\tau^*) =\boldsymbol X^\diese(t),\qquad \forall t\in[0,2\pi],
\end{equation}
where $\boldsymbol  X^\diese$ is defined in the same manner as $\boldsymbol X$, that is by replacing in (\ref{eq:5.0})
$I$ by $I^\diese$ and $\mathbf x_0$ by $\mathbf x_0^\diese$. Furthermore, since these two functions are $2\pi$-periodic,  relation
(\ref{eq:5.4}) holds for the smoothed versions of $\boldsymbol X$ and  $\boldsymbol X^\diese$ as well.

This observation, depicted in Fig.~\ref{fig:4}, leads to the following criterion for keypoint matching based on their LoFT descriptors.
Given a threshold $\lambda>0$ and a (estimated) noise level $\sigma$, we declare that the LoFT descriptors $\bY$ and $\bY^\diese$ corresponding to
the keypoints $\mathbf x_0$ and $\mathbf x_0^\diese$ and defined by (\ref{eq:5.3}) match if and only if
\begin{equation}\label{eq:5.5}
\Delta := \frac1{4\sigma^2} \min_{\tau\in[0,2\pi]}\sum_{j=1}^k \|\bY_j-e^{-\ii j \tau}\bY_j^\diese\|_2^2 \le \lambda.
\end{equation}
According to the theoretical results established in foregoing sections, under the null hypothesis (that is when the pair $(\mathbf x_0,\mathbf x_0^\diese)$
is a true match) the test statistic $\Delta$ is asymptotically parameter free and the limiting distribution is Gaussian with zero mean and a variance that can
be easily computed. We carried out some experiments, reported in the next subsection, which show that this property holds not only for small $\sigma$,
but also for reasonably high values of it. Furthermore, substituting the true noise level by an estimated one yields sensibly similar results.
\begin{figure}[ht]
\hbox to \hsize{\hfill
\includegraphics[width=0.21\textwidth]{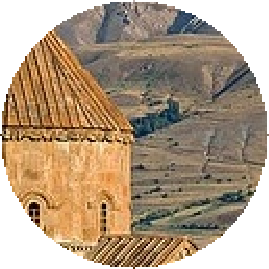}\hfill
\includegraphics[width=0.26\textwidth]{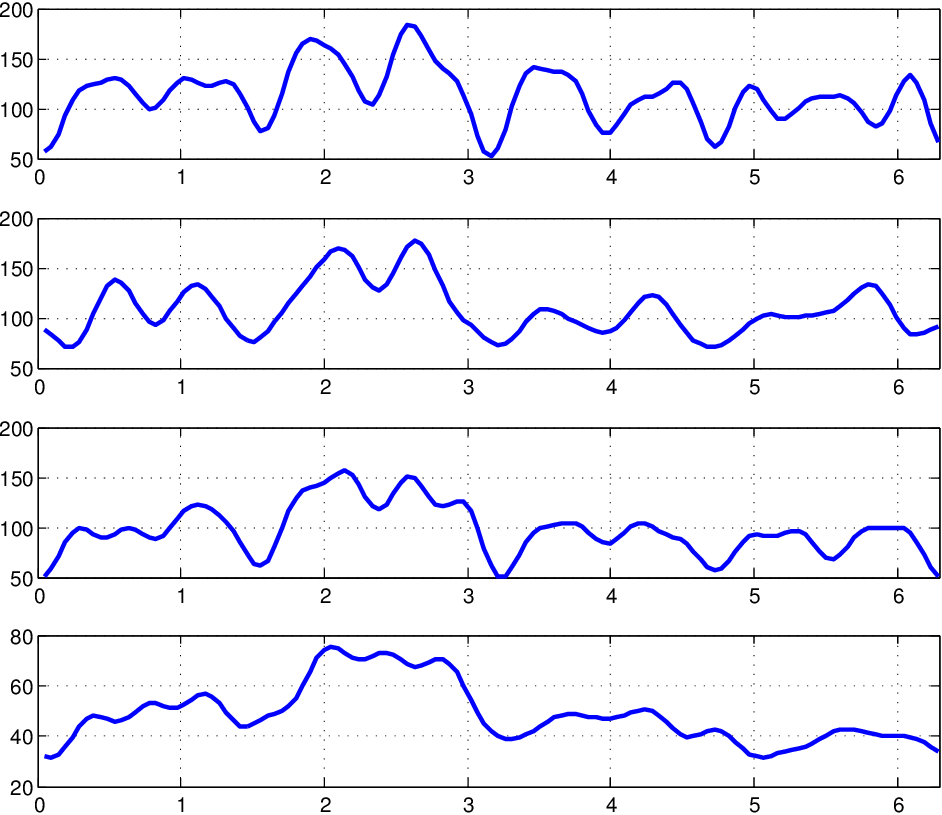}\hfill
\includegraphics[width=0.21\textwidth]{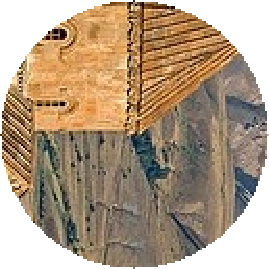}\hfill
\includegraphics[width=0.26\textwidth]{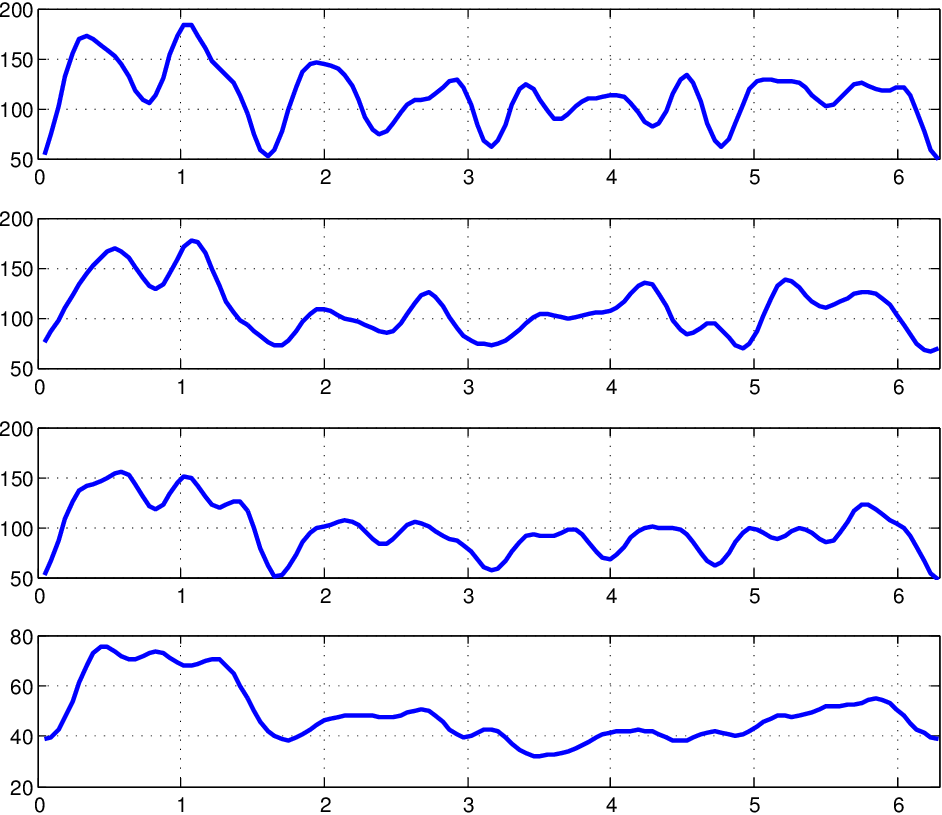}}
\hbox to \hsize{\bf \hfill(a) \hfill\hfill (b)\hfill\hfill (c)\hfill\hfill (d)\hfill}
\caption{Rotated image regions correspond to shifted curves. (a) and (b) are an image region and the corresponding curves used in the definition of
the LoFT descriptor. (c) is the image region (a) rotated by the angle $\pi/4$, the corresponding curves are depicted in (d). We clearly see that the four curves
in (d) are obtained from those in (b) by shifting to left (with a shift equal to $\pi/2 \approx 1.57$. }
\label{fig:4}
\end{figure}

\subsection{Experimental evaluation}

To empirically assess the properties of the LoFT descriptor and the matching criterion defined in (\ref{eq:5.5}), we carried out some numerical experiments.
All the codes and the images necessary to reproduce the results and the figures reported in this section can be freely downloaded from the website \url{http://imagine.enpc.fr/~dalalyan/LoFT}.

We chose two grayscale images of resolution $300\times450$ that coincide up to a rotation by an angle $\pi/2$. We degraded these images by adding two
independent white Gaussian noises of variance $\sigma^2$. This resulted in two noisy images $I$ and $I^\diese$ depicted in Fig.~\ref{fig:5} (left panels).
Then we chose at random $L = 10^4$ pairs of truly matching points $(\mathbf x_\ell,\mathbf x_\ell^\diese)$. The only restriction made on these points is
that the distance between two distinct points $\mathbf x_\ell$ and $\mathbf x_{\ell'}$ is at least of $5$ pixel. Then we chose points
$\tilde{\mathbf x}_{\ell}^\diese$ such that $\tilde{\mathbf x}_{\ell}^\diese={\mathbf x}_{\ell}^\diese + (10,10)$ which we used as false matches with $\mathbf x_\ell$'s. We computed the corresponding LoFT
descriptors in each image. This yielded three sets of vectors $\{\bY_\ell\}$, $\{\bY^\diese_\ell\}$ and $\{\tilde\bY^\diese_\ell\}$. Finally,
the values of the test statistic $\Delta$
were computed for the pairs $(\bY_\ell, \bY^\diese_\ell)$ and $(\bY_\ell, \tilde\bY^\diese_{\ell})$. We obtained two samples $\delta_1,\ldots,\delta_L$ and
$\delta'_1,\ldots,\delta'_L$. The first sample characterizes the behavior of the test statistic under the null (\textit{i.e.}, for true matches),
whereas the second sample characterizes the behavior of the test statistic under the alternative (false matches).

The parallel boxplots of these two samples, computed for several values of $\sigma$, are plotted in the right-bottom panel of Fig.~\ref{fig:5}.
Two scenarios were considered: known $\sigma$ and unknown $\sigma$. In the second scenario the estimator of $\sigma$ proposed by \cite{Immerkaer}
were used and injected in (\ref{eq:5.5}) instead of $\sigma$. The results of the first scenario are plotted in the first row of the right-bottom
panel of Fig.~\ref{fig:5}, while those of the second scenario are in the second row. One can note that the different values for the noise level considered
in these experiments are $\sigma\in\{5,10,30,60\}$.

In the light of these figures, several observations can be made.  Perhaps the most striking one is that even for a noise level as high as\footnote{Note that the
standard deviation of the image intensities in this example being equal to 40.25, the signal-to-noise ratio is very small.} $\sigma=30$, there is a clear
separation between the two samples. Furthermore, the top-right panel of Fig.~\ref{fig:5} shows that the distribution of the test statistic under the null
is extremely close to the Gaussian distribution, as proved in our theoretical results. Therefore, choosing as $\lambda$ any reasonable quantile
($95\%$, $99\%$, $99.9\%$) of this distribution results in rejecting all the false matches. In other terms, the $p$-values associated to the elements
of the second sample, the one of false matches, are all below the level of $0.1\%$.

\begin{figure}[ht!]
\hbox to \hsize{\hfill
\hbox{\includegraphics[width=0.45\textwidth]{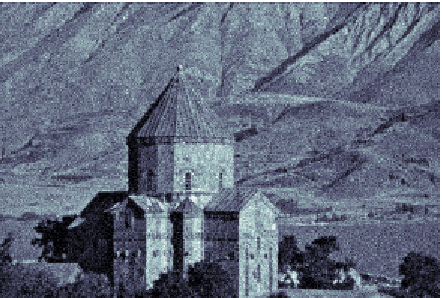}}\hfill
\hfill\includegraphics[width=0.45\textwidth]{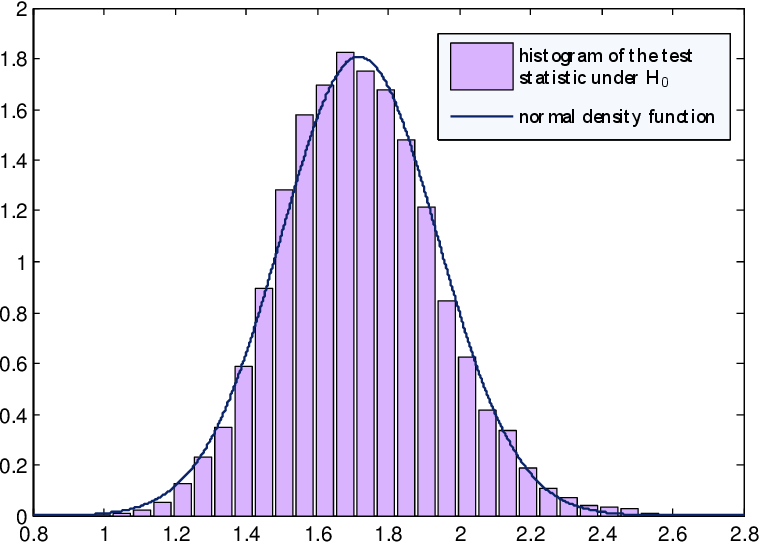}\hfill}
\medskip
\hbox to \hsize{
\hfill\lower-25pt\hbox{\includegraphics[height=0.45\textwidth]{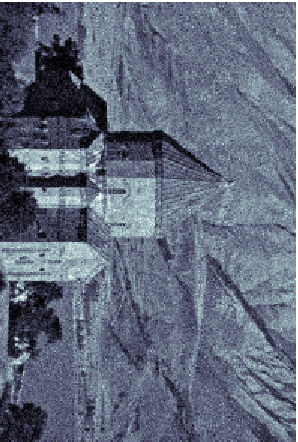}}
\hfill\includegraphics[height=0.55\textwidth]{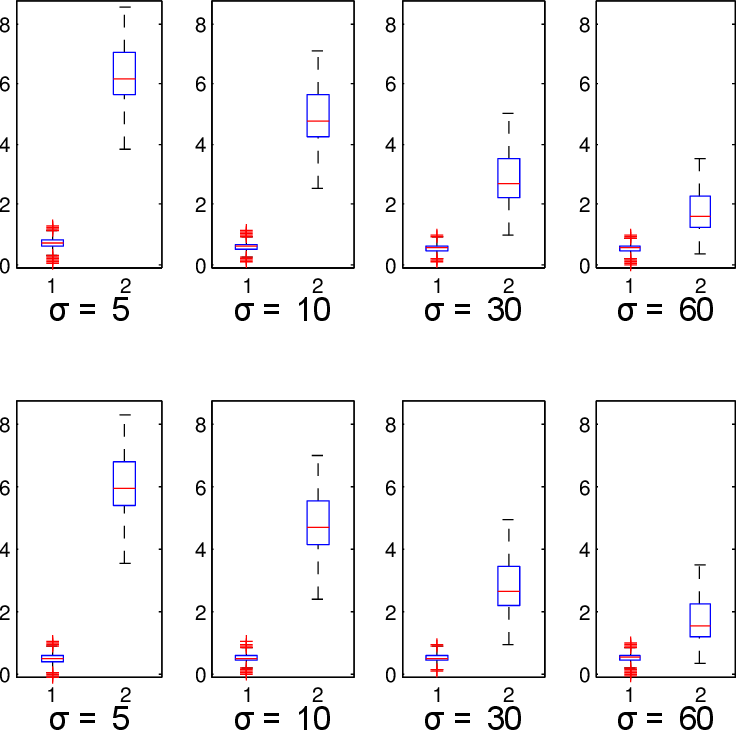}\hfill}
\caption{Experimental evaluation of the discriminative power of the LoFT descriptor. \textit{Top-left and bottom-left}: two noisy images used in our experiments that coincide up to a rotation by angle $\pi/2$. \textit{Top-right}: the histogram of the test statistic $\Delta$ (cf. (\ref{eq:5.5})) computed for $10^4$ randomly chosen pairs of truly matching points. The true noise level is equal to 30, for image intensities ranging from 0 to 255. The value of $\sigma$ used in (\ref{eq:5.5}) is the estimated noise level computed by the procedure described in \citep{Immerkaer}. One can observe that the distribution is very close to a Gaussian one. \textit{Bottom-right}: the boxplots of the logarithm of the test statistic $\Delta$ for true matches and for false matches computed for different noise levels.
In the first (top) row the true $\sigma$ is used in (\ref{eq:5.5}), while in the second row we used the estimator provided
by \cite{Immerkaer}. A remarkable property is that the boxplots under $\mathbf H_0$ are almost not impacted by the change of $\sigma$. It is also noteworthy that the boxplots of true matches are well separated from those of false matches for all values of $\sigma$ except for $\sigma=60$.}
\label{fig:5}
\end{figure}

A second important observation is that the result is almost not impacted by the substitution of the true noise variance by its estimated value. This may
be very useful for applying LoFT descriptors to very noisy images such as those encountered in medical imaging and astrophysics. A last observation is that
when $\sigma=60$, the noise is so strong that nearly $5\%$ of false matches are classified as true matches, when the threshold $\lambda$ in (\ref{eq:5.5})
is chosen equal to $2.22$, which is the $99\%$-quantile of the distribution of the test statistic under the null. This is not so surprising and shows the limits
of the presented approach.

To close this section, let us stress that the primary aim here was to show the applicability and the potential of the proposed approach. A more comprehensive
experimental evaluation of the discriminative power of the LoFT descriptor comparing it to other state-of-the-art descriptors is the subject of an ongoing work.

\section{Conclusion}

In the present work, we provided a methodological and theoretical analysis of the curve registration problem from a
statistical standpoint based on the nonparametric goodness-of-fit testing. In the case where the noise is white Gaussian and
additive with a small variance we established that under the null hypothesis the penalized
log-likelihood ratio statistic is asymptotically distribution free. This result is valid for the weighted $l^2$-penalization
under some mild assumptions on the weights. Furthermore, we proved that the test based on the Gaussian (or chi-squared)
approximation of the penalized log-likelihood ratio  statistic is consistent. These results naturally carry over to other
nonparametric models for which asymptotic equivalence (in the Le Cam sense) with the Gaussian white noise has been proven.

It can be interesting, however, to develop a direct inference in these models. In particular, the model of spatial Poisson processes
(cf.\ \cite{Ingster_Kutoyants}) can be of special interest because of its applications in image analysis. 
Some important issues closely related to the present work have been treated in the companion paper \cite{Collier2012}. In that paper,
focusing on the case of equal variances of noise and considering only projection weights, the minimax rate of separation of the null hypothesis
from the alternative is obtained (up to a $\log$-factor) and an adaptive test is proposed. It should be mentioned, that although the procedure
proposed in \cite{Collier2012} is adaptive and achieves the asymptotically minimax rates of separation, it is not necessarily suitable for the
real applications because it is often overly conservative. This is due to the fact that the minimax procedures are essentially designed to be
optimal in the worst cases, but can be outperformed in the favorable situations.

Finally, we demonstrated that the main ideas introduced in the present work lead to a new local descriptor for digital images tailored to
testing for rotation between two image regions. The optimization of the implementation of this descriptor and its more systematic
evaluation on various benchmark datasets used in computer vision is another promising direction of future research.

\appendix
\renewcommand*{\thesection}{Appendix \Alph{section}}

\section{Proofs of the theorems}\label{sec:4}

This and the following appendices contain the technical proofs of the theorems stated in previous sections.
For the sake of self-containedness, we provide all the details of the proofs although some of them---such as
the Berman theorem---have been already used in earlier references (see, for instance, \cite{Collier2012}).

The proof of Wilks' phenomenon is divided into several parts. First we assume that $\mathbf{H_0}$ is true and study the convergence of the pseudo-estimator $\hat\tau$
(of the shift $\bar\tau^*$) defined as the maximizer of the log-likelihood over the interval $[\bar\tau^*-\pi,\bar\tau^*+\pi]$. Here, $\bar\tau^*$ is an element of $[0,2\pi[$
such that $c_j=e^{-\ii j\bar\tau^*}c_j^\diese$, for all $j\ge 1$.

\subsection{Maximizer of the log-likelihood}

\begin{proposition}\label{prop:1}
Let  $\bc\in\calF_{s,L}$ for some $s\in(0,1]$, $L>0$ and let $|c_1| >0$. If the shrinkage weights $\nu_j$ satisfy conditions {\bf(A)} and {\bf(B)},
then the solution $\hat\tau$ to the optimization problem
\begin{equation*}
\hat{\tau} = \arg\max_{\tau:|\tau-\bar\tau^*|\le \pi} M(\tau),\qquad \text{with}\qquad
M(\tau) = \sum_{j\ge 1} \nu_j \Re (e^{\ii j\tau} Y_j \overline{\Ydiese_j})
\end{equation*}
satisfies the asymptotic relation
\begin{equation*}
|\hat{\tau}-\bar\tau^*| = \sigma_*\sqrt{\log N_{\sigma_*}}\big(N_{\sigma_*}^{1-s}+\sigma_* N_{\sigma_*}^{3/2}\big)O_P(1),\qquad\text{as}\quad\sigma_*\to 0.
\end{equation*}
\end{proposition}

\begin{proof}[Proof of Proposition~\ref{prop:1}]
Throughout this proof, we work under the null hypothesis $\mathbf{H_0}$. If we set $\eta_j=e^{-\ii j\bar\tau^*}{\epsilon_j}$ and $\etadiese_j=\epsdiese_j$, we can write the decomposition
\begin{align*}
M(\tau)  - \esp[M(\tau)] = \sigma_* S(\tau)+ \sigma\sigmadiese D(\tau+\bar\tau^*),
\end{align*}
where
\begin{align*}
S(\tau)  = \sum_{j\geq1} \nu_j \Re\Big\{ e^{\ii j\tau}\Big(\frac{\sigma}{\sigma_*}\overline{c_j}\eta_j+\frac{\sigmadiese}{\sigma_*} c_j\overline{\etadiese_j}\Big)\Big\} ,\quad
D(\tau)    = \sum_{j\geq1} \nu_j \Re \big(e^{\ii j\tau}\eta_j \overline{\etadiese_j}\big).
\end{align*}
Furthermore, the expectation of $M(\tau)$ is given by $\esp[M(\tau)] = \sum_{j\geq1} \nu_j |c_j|^2 \cos[j(\tau- \bar\tau^*)]$. In what follows, for every function $f:\RR\to\RR$,
we denote by $\|f\|_\infty$ the supremum over $\RR$ of the function $f$.

On the one hand, using the assumption $|c_1|>0$ along with condition (A), we get that for every $\tau\in[\bar\tau^*-\pi,\bar\tau^*+\pi]$ it holds
\begin{align*}
\frac{\esp\big[M(\tau)\big]-\esp\big[M(\bar\tau^*)\big]}{(\tau-\bar\tau^*)^2} &\leq -\nu_1 |c_1|^2 \frac{1-\cos(\tau-\bar\tau^*)}{(\tau-\bar\tau^*)^2}
\le -\frac{2|c_1|}{\pi^2} \triangleq C < 0.
\end{align*}
Therefore,
\begin{align*}
M(\tau)-M(\bar\tau^*) &= \esp[M(\tau)] - \esp[M(\bar\tau^*)] + \sigma_* \big[ S(\tau)-S(\bar\tau^*)\big] + \sigma\sigmadiese \big[ D(\tau)-D(\bar\tau^*)\big]\\
            & \le -C\,|\tau-\bar\tau^*|^2 + \sigma_* |\tau-\bar\tau^*|\cdot\|S'\|_\infty + \sigma_*^2 |\tau-\bar\tau^*|\cdot\|D'\|_\infty\\
            & = |\tau-\bar\tau^*|\big\{\sigma_* \|S'\|_\infty + \sigma_*^2 \|D'\|_\infty-C|\tau-\bar\tau^*|\big\}.
\end{align*}
Replacing in this inequality $\tau$ by $\hat\tau$ and using that $M(\hat\tau)-M(\bar\tau^*)\ge 0$, we get
\begin{equation}\label{eq:14}
|\hat\tau-\bar\tau^*| \le C^{-1}\big\{\sigma_* \|S'\|_\infty + \sigma_*^2 \|D'\|_\infty\big\}.
\end{equation}
On the other hand, we have
\begin{equation*}
S'(\tau) = \frac{\sqrt{\sigma^2+{\sigmadiese}^2}}{\sigma_*}\sum_{j\geq1} j |c_j| \nu_j \Re\big( e^{\ii j\tau}\zeta_j \big),
\end{equation*}
where $\zeta_j$ are i.i.d. complex valued random variables, whose real and imaginary parts are independent standard Gaussian random variables.
Therefore, the large deviations of the sup-norm of $S'$ can be controlled using the following lemma.

\begin{lemma}\label{lem:2.5} Let $\bs=(s_1,\ldots,s_N)$ be a vector from $\RR^{N}$ and let  $\{\xi_j\}$ and $\{\xi'_j\}$
be two independent sequences of i.i.d.\ $\calN(0,1)$ random variables.
The sup-norm of the function $Z(t) = \sum_{j=1}^N s_j \{\cos(jt)  \xi_j+\sin(jt)  \xi'_j\}$, satisfies
\begin{equation*}
\prob( \|Z\|_\infty \geq \|\bs\|_2 x) \leq (N+1) e^{-x^2/2},\qquad \forall x>0.
\end{equation*}
\end{lemma}
\begin{proof}
See \ref{sec:5}.
\end{proof}

To apply this result to $Z(t)=S'(t)$, we choose $s_j= \frac{\sqrt{\sigma^2+{\sigmadiese}^2}}{\sigma_*}j |c_j|\nu_j$, which leads to a vector
$\bs=(s_1,\ldots,s_{N_{\sigma_*}})$ with Euclidean norm
$$
\|\bs\|_2^2=\frac{\sigma^2+{\sigmadiese}^2}{\sigma_*^2} \sum_{j=1}^{N_{\sigma_*}}j^2 |c_j|^2\nu_j^2\le 2 N_{\sigma_*}^{2(1-s)} L^2.
$$
The last inequality follows from the fact that $\bc\in\calF_{s,L}$, $\sigma_*=\max(\sigma,\sigmadiese)$ and $\nu_j\in[0,1]$ for every $j$.
Using this bound, Lemma~\ref{lem:2.5} and the fact that $N_{\sigma_*}\ge 1$, we get that the inequality
\begin{equation}\label{eq:15}
\prob\Big(\|S'\|_\infty \ge 2LN_{\sigma_*}^{1-s}\sqrt{\log(4N_{\sigma_*}/\alpha)}\Big)\le \frac{\alpha}{2}
\end{equation}
holds true for every $\alpha\in(0,1)$.

Finally, the large deviations of the term $\|D'\|_\infty$ are controlled by the following lemma.

\begin{lemma}\label{lem:5}
Let $N$ be some positive integer and let $\eta_j$, $\etadiese_j$, $j=1,\ldots,N$ be independent complex valued random variables such
that their real and imaginary parts are independent standard Gaussian variables. Let $\bs=(s_1,\ldots,s_N)$ be a vector of real numbers.
Denote $Z(t) = \sum_{j=1}^{N} s_j \Re\big(e^{\ii jt}\eta_j\etadiese_j\big)$ for every $t$ in $[0,2\pi]$ and $\|Z\|_\infty=
\sup_{t\in[0,2\pi]} |Z(t)|$. Then,
\begin{equation*}
\prob\Big\{\|Z\|_\infty> \sqrt{2} x\big(\|\bs\|_2+y\|\bs\|_\infty\big) \Big\} \le (N+1)e^{-x^2/2}+e^{-y^2/2}, \qquad \forall x,y>0.
\end{equation*}
\end{lemma}
\begin{proof}
See~\ref{sec:5}.
\end{proof}

In order to bound the sup-norm of  $D'(\cdot)$ using Lemma~\ref{lem:5}, we set $N=N_{\sigma_*}$ and $s_j=j\nu_j$ for all $j=1,\ldots,N_{\sigma_*}$.
This yields $\|\bs\|_2\le N_{\sigma_*}^{3/2}$ and $\|\bs\|_\infty\le N_{\sigma_*}$. Therefore,
\begin{equation}\label{eq:16a}
\prob\Big\{\|D'\|_\infty> \sqrt{2} xN_{\sigma_*}\big(\sqrt{N_{\sigma_*}}+y\big) \Big\} \le (N_{\sigma_*}+1)e^{-x^2/2}+e^{-y^2/2}, \qquad \forall x,y>0.
\end{equation}
For any $\alpha\in(0,1)$, choosing $x=\sqrt{2\log(8N_{\sigma_*}/\alpha)}$ and $y=\sqrt{2\log(4/\alpha)}$, we arrive at
\begin{equation}\label{eq:16b}
\prob\Big\{\|D'\|_\infty> 2 N_{\sigma_*}\sqrt{\log(8N_{\sigma_*}/\alpha)}\big(\sqrt{N_{\sigma_*}}+\sqrt{\log(4/\alpha)}\big) \Big\} \le \frac{\alpha}{2}.
\end{equation}
Inequalities (\ref{eq:14}), (\ref{eq:15}) and (\ref{eq:16b}) imply that  $|\hat\tau-\bar\tau^*|$ is, in probability, at most of the order
$\sigma_*\sqrt{\log N_{\sigma_*}}\big(N_{\sigma_*}^{1-s}+\sigma_* N_{\sigma_*}^{3/2}\big)$.
\end{proof}

\subsection{Proof of Theorem~\ref{th:1}}\label{A.2}

Recall that we present the proof in the case $\bc\in\calF_{s,L}$ for $s>7/8$. The claim of Theorem~\ref{th:1} can be readily obtained
by taking $s=1$.

One can check that, under $\mathbf{H_0}$,
\begin{align}\label{eq:19}
\Delta_{\bnu,\sigma_*}(\bY^{\bullet,\diese})
					&= \frac{1}{2(\sigma^2+(\sigmadiese)^2)}\min_{\tau\in[0,2\pi[} \Bigg[\sum_{j=1}^{+\infty} \nu_j \big| {Y_j - e^{-\ii j\tau} \Ydiese_j} \big|^2 \Bigg]\\
	       	&= \frac{1}{2(\sigma^2+(\sigmadiese)^2)} \min_{|\tau-\bar\tau^*|\le\pi} \big\{D_{\sigma_*}(\tau)+2C_{\sigma_*}(\tau)+P_{\sigma_*}(\tau)\big\},
\end{align}
where we have used the notation:
\begin{align*}
D_{\sigma_*}(\tau) &= \sum_{j=1}^{+\infty}\limits \nu_j |c_j|^2 \big| {1 - e^{-\ii j(\tau-\bar\tau^*)}}\big|^2, & \text{(deterministic term)}\\
C_{\sigma_*}(\tau) &= \sum_{j=1}^{+\infty}\limits  \nu_j \Re \big[c_j\big(1 - e^{-\ii j(\tau-\bar\tau^*)}\big)
\big(\overline{\sigma\epsilon_j-e^{-\ii j\tau}\sigmadiese\epsdiese_j}\big)\big],& \text{(cross term)}\\
P_{\sigma_*}(\tau) &= \sum_{j=1}^{+\infty}\limits \nu_j \big|\sigma\epsilon_j - e^{-\ii j\tau}\sigmadiese\epsdiese_j\big|^2.& \text{(principal term)}
\end{align*}
(Since $\mathbf{H_0}$ is assumed satisfied, there exists $\bar\tau^*\in[0,2\pi[$ such that $c_j=e^{-\ii j\bar\tau^*}c^\diese_j$ for all $j\ge 1$.)
We denote by $\hat\tau$ the pseudo-estimator of $\bar\tau^*$ defined as the minimizer of the right-hand side of (\ref{eq:19}) over the interval
$[\bar\tau^*-\pi,\bar\tau^*+\pi]$ and study the asymptotic behavior of the terms $D_{\sigma_*}$, $C_{\sigma_*}$ and $P_{\sigma_*}$ separately.

For the deterministic term, in view of Proposition~\ref{prop:1}, it holds that
\begin{align*}
|D_{\sigma_*}(\hat\tau)| &\leq \sum_{j=1}^{+\infty} j^2\nu_j |c_j|^2 {(\hat\tau-\bar\tau^*)^2}
                     \leq {(\hat\tau-\bar\tau^*)^2} \sum_{j=1}^{N_{\sigma_*}} j ^2 |c_j|^2\le N_{\sigma_*}^{2(1-s)}
										L^2(\hat\tau-\bar\tau^*)^2\nonumber\\
                  &=\{\sigma_*^2N_{\sigma_*}^{2(1-s)}(N_{\sigma_*}^{2(1-s)}+\sigma_*^2N_{\sigma_*}^3)\log N_{\sigma_*}\}\,O_p(1).
\end{align*}
Therefore,
\begin{align}\label{eq:10a}
\frac{|D_{\sigma_*}(\hat\tau)|}{\sigma_*^2\|\bnu\|_2} &\leq
                  \{(N_{\sigma_*}^{\frac72-4s}+\sigma_*^2N_{\sigma_*}^{\frac92-2s})\log N_{\sigma_*}\}\,O_p(1).
\end{align}
Let us turn now to the cross term. As $C_{\sigma_*}(\bar\tau^*) = 0$, we have
\begin{equation*}
|C_{\sigma_*}(\hat\tau)| \leq  |\hat\tau-\bar\tau^*| \cdot \|C_{\sigma_*}'\|_\infty.
\end{equation*}
Furthermore, if we define  $\xi_j$ and $\xi_j'$ as respectively the real and the imaginary parts of the random variable
$\frac{c_j}{|c_j|\sqrt{\sigma^2+\sigmadiese{}^2}}\big({\sigma\bar{\epsilon}_j-e^{\ii j\bar\tau^*}\sigmadiese\bar{\epsdiese_j}}\big)$, we get
\begin{align*}
C'_{\sigma_*}(\bar\tau^*+t) &= \sqrt{\sigma^2+\sigmadiese{}^2}\sum_{j=1}^{+\infty} j\nu_j |c_j|\big[\sin(jt)\xi_j+ \cos(jt)\xi_j'\big] ,\qquad\forall t\in\RR.
\end{align*}
Using Lemma~\ref{lem:2.5}, we check that $\|C_{\sigma_*}'\|_\infty$ is, in probability, of the order  $\{\sigma_*N_{\sigma_*}^{1-s}\sqrt{\log N_{\sigma_*}}\}$. Therefore, in view of Proposition~\ref{prop:1},
it holds that
$$
|C_{\sigma_*}(\hat\tau)| = \{\sigma_*^2(N_{\sigma_*}^{1-s}+\sigma_* N_{\sigma_*}^{3/2})N_{\sigma_*}^{1-s}\log N_{\sigma_*}\}\,O_p(1)
$$
and, therefore,
\begin{align}
\frac{|C_{\sigma_*}(\hat\tau)|}{\sigma_*^2\|\bnu\|_2} =
\{(N_{\sigma_*}^{\frac32-2s}+\sigma_* N_{\sigma_*}^{2-s})\log N_{\sigma_*}\}\,O_p(1).
\end{align}

Let us study the last term, $P_{\sigma_*}(\tau) = \sum_{j=1}^{+\infty} \nu_j \big|{\sigma\epsilon_j - e^{-\ii j\tau}\sigmadiese\epsdiese_j} \big|^2$, which will determine the asymptotic behavior of the test statistic. Denoting
$\eta_j=e^{\ii j\bar\tau^* } {\epsilon_j}$ and $\etadiese_j={\epsdiese_j}$,
we can rewrite this term as $P_{\sigma_*}(\tau) = \sum_{j=1}^{+\infty} \nu_j \big|{\sigma\eta_j - e^{-\ii j(\tau-\bar\tau^*)}\sigmadiese\etadiese_j} \big|^2$.
We wish to prove that under the null hypothesis $\mathbf{H_0}$, if conditions {\bf(A)}, {\bf(B)}, $N_{\sigma_*}\to +\infty$ and $\sigma_*^2 N_{\sigma_*}^{5/2}\log(N_{\sigma_*}) = o_P(1)$
are fulfilled, then
\begin{equation*}
H_{\sigma_*}(\hat\tau) = \frac{P_{\sigma_*} (\hat\tau) - 2(\sigma^2+(\sigmadiese)^2)\|\bnu\|_1 }{2(\sigma^2+(\sigmadiese)^2)\|\bnu\|_2} \convloi \mathcal{N}(0,1).
\end{equation*}
To check this property, we decompose the principal term as follows:
\begin{align*}
H_{\sigma_*}(\hat\tau) &= H_{\sigma_*}(\bar\tau^*)+R_{\sigma_*}(\hat\tau),\qquad\text{with}\qquad
R_{\sigma_*}(\hat\tau)=\frac{P_{\sigma_*} (\hat\tau)-P_{\sigma_*} (\bar\tau^*)}{2(\sigma^2+(\sigmadiese)^2)\|\bnu\|_2}.
\end{align*}
We start by writing $H_{\sigma_*}(\bar\tau^*)$ as
\begin{equation*}
H_{\sigma_*}(\bar\tau^*) = \sum_{j=1}^{N_{\sigma_*}} X_{j,\sigma_*},\qquad \text{ with }\qquad X_{j,\sigma_*} = \frac{ \nu_j}{2\|\bnu\|_2}\Big (\Big|\frac{\sigma\eta_j-\sigmadiese\etadiese_j}{(\sigma^2+(\sigmadiese)^2)^{1/2}}\Big|^2-2\Big) ,
\end{equation*}
and, by applying the Berry-Esseen inequality \cite[Theorem 5.4]{Petrov}, which is valid since  the $X_{j,\sigma_*}$'s are independent random variables with mean $0$ and finite third moment.
Furthermore, we have
$$
B_{\sigma_*} = \sum_{j=1}^{N_{\sigma_*}} \var \big( X_{j,\sigma_*} \big) = 1,\qquad\text{ and}\qquad
L_{\sigma_*} = B_{\sigma_*}^{-\frac{3}{2}} \, \sum_{j=1}^{N_{\sigma_*}} \esp |X_{j,\sigma_*}|^3 \le C \, N_{\sigma_*}^{-\frac{1}{2}}.
$$
Therefore, the Berry-Esseen inequality yields
$\sup_x |F_{\sigma_*}(x)-\Phi(x)| \leq K\,L_{\sigma_*}$,
where $\Phi$ is the c.d.f.\ of the standard Gaussian distribution,
$F_{\sigma_*}(x) = \prob\big(B_{\sigma_*}^{-\frac{1}{2}} \sum_{j=1}^{N_{\sigma_*}} X_{j,\sigma_*} < x\big)$
and $K$ is an absolute constant. Hence
\begin{equation*}
H_{\sigma_*}(\bar\tau^*) \convloi \mathcal{N}(0,1).
\end{equation*}
It remains to prove that $R_{\sigma_*}(\hat\tau)$ tends to $0$ in probability, which---in view of Slutski's lemma---will be sufficient for completing the proof.
By the mean value theorem, there exists some real number $t$  between $\hat\tau$ and $\bar\tau^*$ such that
\begin{align*}
R_{\sigma_*}(\hat\tau)
    &= \frac{\sigma\sigmadiese}{\sigma^2+(\sigmadiese)^2}\sum_{j=1}^{+\infty} \frac{\nu_j}{\|\bnu\|_2}
        \Re \big(\eta_j\overline{\etadiese_j}\,(e^{\ii j(\hat\tau-\bar\tau^*)}-1) \big)\\
    &= \frac{\sigma\sigmadiese}{\sigma^2+(\sigmadiese)^2} \sum_{j=1}^{N_{\sigma_*}}
        \frac{j\nu_j(\hat\tau-\bar\tau^*)}{\|\bnu\|_2}\Re\big(e^{\ii jt} \eta_j\overline{\etadiese_j}\big).
\end{align*}
Then, by virtue of Lemma~\ref{lem:5},
\begin{align*}
|R_{\sigma_*}(\hat\tau) | &\le \frac{\sigma\sigmadiese |\hat\tau-\bar\tau^*|}{(\sigma^2+(\sigmadiese)^2)\|\bnu\|_2}
\sup_{t\in[0,2\pi]}\Big|\sum_{j=1}^{N_{\sigma_*}} j\nu_j\Re\big(e^{\ii jt} \eta_j\overline{\etadiese_j}\big)\Big|\\
&=\{\sigma_*(N_{\sigma_*}^{1-s}+\sigma_* N_{\sigma_*}^{3/2})N_{\sigma_*}\log N_{\sigma_*}\}\cdot O_P(1).
\end{align*}
Combining all these relations, we get
\begin{align*}
\frac{(\Delta_{\bnu,\sigma_*}(\bY^{\bullet,\diese})-\|\bnu\|_1)}{\|\bnu\|_2}
		&= \frac{D_{\sigma^*}(\hat\tau) +2C_{\sigma_*}(\hat\tau)}{2(\sigma^2+\sigmadiese{}^2)\|\bnu\|_2}+H_{\sigma_*}(\bar\tau^*)+
		R_{\sigma_*}(\hat\tau)\\
		&= H_{\sigma_*}(\bar\tau^*)+\{N_{\sigma_*}^{\frac72-4s}+\sigma_*^2N_{\sigma_*}^{\frac92-2s}+
		N_{\sigma_*}^{\frac32-2s}+\sigma_* N_{\sigma_*}^{2-s}+\sigma_*^2 N_{\sigma_*}^{\frac52}\}O_P(\log N_{\sigma_*}).
\end{align*}
Under the assumptions $s>7/8$, $N_{\sigma_*}\to\infty$ and $\sigma_*^2N_{\sigma_*}^{-2s+9/2}\log N_{\sigma_*}\to 0$, as $\sigma_*\to 0$,
we infer the
relation ${(\Delta_{\bnu,\sigma_*}(\bY^{\bullet,\diese})-\|\bnu\|_1)}/{\|\bnu\|_2} = H_{\sigma_*}(\bar\tau^*) + O_P(1)$.
The application of the Slutsky lemma completes the proof.

\subsection{Power of the test}

The aim of this section is to present a proof of Theorem~\ref{th:2}. To this end, we study the test statistic
$T_{\sigma_*}={(\Delta_{\bnu,\sigma_*}(\bY^{\bullet,\diese})-\|\bnu\|_1)}/{\|\bnu\|_2}$, and show that it tends to $+\infty$ in probability under $\mathbf{H_1}$.
Actually, the hypothesis $\mathbf{H_1}$ will be supposed to be satisfied throughout this section. It holds true that:
\begin{align*}
\Delta_{\bnu,\sigma_*}(\bY^{\bullet,\diese}) &= \frac{1}{\sigma_*^2} \min_{\tau\in[0,2\pi]} \sum_{j\geq 1} \nu_j
\big| Y_j - e^{-\ii j\tau} \Ydiese_j\big|^2 \\
         &= \frac{1}{\sigma_*^2} \min_{\tau\in[0,2\pi]} \sum_{j\geq1} \nu_j \Big| \big(c_j- e^{-\ii j\tau}\cdiese_j) +
            (\sigma\epsilon_j- e^{-\ii j\tau}\sigma^\diese\epsdiese_j\big)\Big|^2 \\
         &\geq \frac{1}{\sigma_*^2} \min_{\tau\in[0,2\pi]} \Big\{ \sum_{j\geq1} \nu_j |c_j- e^{-\ii j\tau}\cdiese_j|^2\Big\}
               -\frac{2}{\sigma_*} \max_{\tau\in[0,2\pi]} \Big\{\sum_{j\geq1} \nu_j |c_j- e^{-\ii j\tau}\cdiese_j|\cdot
               (|\epsilon_j|+ |\epsdiese_j|) \Big\}.
\end{align*}
Let us focus on the first term. Denoting $\delta_{\sigma_*}=\min\{j\geq1,\nu_j<\overline c\}$, we get by condition {\bf (C)} that $\delta_{\sigma_*}\to+\infty$, which implies 	
\begin{align*}
\min_{\tau\in[0,2\pi]} \sum_{j\geq1} \nu_j |c_j- e^{-\ii j\tau}\cdiese_j|^2
         &\geq {\overline c} \min_{\tau\in[0,2\pi]}\sum_{j=1}^{\delta_{\sigma_*}} |c_j- e^{-\ii j\tau}\cdiese_j|^2 \\
         &\geq {\overline c} \Big( \min_{\tau\in[0,2\pi]} \sum_{j=1}^{+\infty} |c_j- e^{-\ii j\tau}\cdiese_j|^2 - 4L^2\delta_{\sigma_*}^{-2}\Big)\\
         &\geq {\overline c} \big( \rho - 4L^2\delta_{\sigma_*}^{-2}\big).
\end{align*}
For the second term, we use that for every $\tau$ it holds that
\begin{align*}
\sum_{j\geq1} \nu_j \big|c_j- e^{-\ii j\tau}\cdiese_j\big|\cdot\big|\epsilon_j- e^{-\ii j\tau}\epsdiese_j\big|
&\leq \max_{j=1,\ldots,N_{\sigma_*}} (|\epsilon_j|\vee |\epsilon_j^\diese|) \sum_{j=1}^{N_{\sigma_*}} \Big(|c_j|+|c_j^\diese|\Big) \\
&= O_P(\sqrt{\log N_{\sigma_*}})\sum_{j=1}^{N_{\sigma_*}} \Big(|c_j|+|c_j^\diese|\Big).
\end{align*}
Furthermore, using the Cauchy-Schwarz inequality, we get
\begin{align*}
\sum_{j=1}^{N_{\sigma_*}} \Big(|c_j|+|c_j^\diese|\Big)\le \bigg(\sum_{j=1}^{N_{\sigma_*}} j^{-2}\bigg)^{1/2}
     \bigg(\sum_{j=1}^{N_{\sigma_*}} j^2 \Big(|c_j|+|c_j^\diese|\Big)^2\bigg)^{1/2}\le 3 L.
\end{align*}
Therefore,
\begin{align*}
\max_{\tau\in[0,2\pi]} \sum_{j\geq 1} \nu_j \big|c_j- e^{-\ii j\tau}\cdiese_j\big|\cdot\big|\epsilon_j- e^{-\ii j\tau}\epsdiese_j\big|
&= O_P(\sqrt{\log N_{\sigma_*}}).
\end{align*}
Combining all these estimates, we get
\begin{align*}
T_{\sigma_*}&=\frac{\Delta_{\bnu,\sigma_*}(\bY^{\bullet,\diese})-\|\bnu\|_1}{\|\bnu\|_2} \geq \frac{\overline{c} \rho-
4L\overline{c}\delta_{\sigma_*}^{-2} -O_P \Big(\sigma_*\sqrt{\log N_{\sigma_*}}\Big)-\sigma_*^2 N_{\sigma_*}}{\sigma_*^2 \sqrt{N_{\sigma_*}}} \convproba +\infty
\end{align*}
in view of the convergences $\delta_{\sigma_*}\to\infty$, $N_{\sigma_*}\to\infty$ and $\sigma^2_*N_{\sigma_*}\to 0$ as $\sigma_*\to 0$.

\section{Proof of Theorem~\ref{th:3}}

Throughout this proof, we assume that $\mathbf{H_0}$ is fulfilled, that is for some $\bar\tau^*\in [0,2\pi[$, the relation $c_j = e^{-\ii j\bar\tau^*} c_j^\diese$ holds for every $j\in\NN$. We also recall that $\sigma$ and $\sigmadiese$ are assumed to be equal.
One easily checks that
\begin{align}\label{pr4:1}
\min_{\tau\in[0,2\pi[} \frac{\|\bY-\ee(\tau)\circ\bY^\diese\|_{2,\nnu}^2}{2(\sigma^2+\sigma^\diese{}^2)}
    &\le \frac{\|\bY-\ee(\bar\tau^*)\circ\bY^\diese\|_{2,\nnu}^2}{2(\sigma^2+\sigma^\diese{}^2)}= \|\nnu\|_1 + \frac12\sum_{j=1}^p \nu_j (|\xi_j|^2-2).
\end{align}
where $\xi_j = (\epsilon_j- e^{-\ii j \bar\tau^*}\epsilon_j^\diese)/\sqrt{2}$ for every $j\in\NN$.
Therefore, using the notation $\hat U_{\sigma_*} = \frac{4\sigma_*^2\|1-\nnu\|_1}{\|\bY\|_{2,1-\nnu}^2+\|\bY^\diese\|_{2,1-\nnu}^2} $
and $Z_{\sigma_*}=\sum_{j=1}^p \frac{\nu_j}{2\|\nnu\|_2} (|\xi_j|^2-2)$, we get
\begin{align}\label{pr4:2}
T(\bY^{\bullet,\diese})
    & =  \frac{\hat U_{\sigma_*}}{\|\nnu\|_2}\min_{\tau\in[0,2\pi[} \frac{\|\bY-\ee(\tau)\circ\bY^\diese\|_{2,\nnu}^2}{4\sigma_*^2}
		-\frac{\|\nnu\|_1}{\|\nnu\|_2} \nonumber\\
    &\le   \hat U_{\sigma_*} \bigg(\frac{\|\nnu\|_1}{\|\nnu\|_2} +
    \sum_{j=1}^p \frac{\nu_j}{2\|\nnu\|_2} (|\xi_j|^2-2)\bigg)-\frac{\|\nnu\|_1}{\|\nnu\|_2} \nonumber\\
    &=  Z_{\sigma_*} +\big(\hat U_{\sigma_*} -1\big)\bigg(\frac{\|\nnu\|_1}{\|\nnu\|_2} +
    Z_{\sigma_*}\bigg).
\end{align}
Since the random variables $Z_{\sigma_*}$ tend in distribution to $\mathcal N(0,1)$ as $\sigma_*\to 0$,
and $\frac{\|\nnu\|_2}{\|\nnu\|_1} \le \frac{\|\nnu\|_2}{\|\nnu\|_2^2} \le (\underline{c} N_{\sigma_*})^{-1/2}$ (the last inequality follows from condition \textbf{(B)}), we arrive at
\begin{align}\label{pr4:3}
T(\bY^{\bullet,\diese})
    &\le Z_{\sigma_*}  +\big(\hat U_{\sigma_*} -1\big)\frac{\|\nnu\|_1}{\|\nnu\|_2}\big(1+O_P(N_{\sigma_*}^{-1/2})\big).
\end{align}
To complete the proof, we study the behavior of $\hat U_{\sigma_*}$ as $\sigma_*\to 0$. To this end, we define
\begin{align*}
Z_{\sigma_*}' & = \frac{(2\|1-\nnu\|_1-\|\eeps\|_{2,1-\nnu}^2)+(2\|1-\nnu\|_1-\|\eeps^\diese\|^2_{2,1-\nnu})}{4\|1-\nnu\|_2},\\
Z_{\sigma_*}''& = \frac{\langle \bc, \eeps + \ee(\bar\tau^*)\circ\eeps^\diese \rangle_{1-\nnu}}{\sqrt{2}\|\bc\|_{2,(1-\nnu)^2}}.
\end{align*}
We have
\begin{align}\label{pr4:4}
\hat U_{\sigma_*}^{-1}
    &= 1+ \frac{\|\bY\|_{2,1-\nnu}^2+\|\bY^\diese\|_{2,1-\nnu}^2 -4\sigma_*^2\|1-\nnu\|_1}{4\sigma_*^2\|1-\nnu\|_1}\nonumber\\
    &= 1- \frac{\|1-\nnu\|_2}{\|1-\nnu\|_1}\;Z_{\sigma_*}' +\frac{\|\bc\|_{2,1-\nnu}^2}{2\sigma_*^2\|1-\nnu\|_1} +
		\frac{\|\bc\|_{2,(1-\nnu)^2} Z_{\sigma_*}'' }{\sqrt{2}\sigma_*\|1-\nnu\|_1}.
\end{align}
To evaluate the right-hand side, we first upper-bound  $\|\bc\|_{2,1-\nnu}^2 = \sum_j (1-\nu_j) |c_j|^2$ by the expression
$\big[\max_{j\ge 1} j^{-2}(1-\nu_j)\big] \sum_j j^2 |c_j|^2 \le c' L N_{\sigma_*}^{-2}$. Taking into account the facts that
$\sigma_*N_{\sigma_*}= O(1)$,  $Z_{\sigma_*}'' \sim \mathcal N(0,2)$ and $\|\bc\|_{2,(1-\nnu)^2}^2 = \sum_j (1-\nu_j)^2 |c_j|^2 \le  \|\bc\|_{2,1-\nnu}^2 \le c' L N_{\sigma_*}^{-2} $, we get
\begin{align*}
\hat U_{\sigma_*}^{-1}
    &= 1- \frac{\|1-\nnu\|_2}{\|1-\nnu\|_1}\;Z_{\sigma_*}' + \frac{O((\sigma_*N_{\sigma_*})^{-2}) + O_P((\sigma_*N_{\sigma_*})^{-1})}{\|1-\nnu\|_1}\nonumber\\
    &= 1- \frac{\|1-\nnu\|_2}{\|1-\nnu\|_1}\;Z_{\sigma_*}'+\frac{O_P(1)}{(\sigma_*N_{\sigma_*})^2\|1-\nnu\|_1}.
\end{align*}
Finally, using the inequality $\|1-\nnu\|_1\ge p-N_{\sigma_*}$, we get that
\begin{align}\label{pr4:5}
\hat U_{\sigma_*}^{-1} &= 1- \frac{\|1-\nnu\|_2}{\|1-\nnu\|_1}Z_{\sigma_*}'+ \frac{O_P(1)}{(p-N_{\sigma_*})(\sigma_*N_{\sigma_*})^2}.
\end{align}
On the other hand, since $\nu_j\in[0,1]$ for every $j$ and $\nu_j=0$ for $j\ge N_{\sigma_*}$, we have
$\|1-\nnu\|_2\le \|1-\nnu\|_1^{1/2}  = \frac{\|1-\nnu\|_1}{\|1-\nnu\|_1^{1/2}} \le \|1-\nnu\|_1/(p-N_{\sigma_*})^{1/2}$.
In particular, relation (\ref{pr4:5}) in conjunction with the condition $(p-N_{\sigma_*})\sigma_*^2N_{\sigma_*}^{3/2}\to+\infty$ implies that
$|\hat U_{\sigma_*}^{-1}-1| \le \frac{\|1-\nnu\|_2}{\|1-\nnu\|_1}\;|Z_{\sigma_*}'|+o_P(N_{\sigma_*}^{-1/2})
= (p-N_{\sigma_*})^{-1/2}O_P(1)+o_P(N_{\sigma_*}^{-1/2})$ and, therefore, $\hat U_{\sigma_*} = 1+ O_P(N_{\sigma_*}^{-1/2})$.
Combining this with (\ref{pr4:5}) and (\ref{pr4:3}), we arrive at
\begin{align*}
T(\bY^{\bullet,\diese})
    &\le  Z_{\sigma_*}  +\hat U_{\sigma_*}\big(1-\hat U_{\sigma_*}^{-1}\big)\frac{\|\nnu\|_1}{\|\nnu\|_2}\big(1+O_P(N_{\sigma_*}^{-1/2})\big)\nonumber\\
    &=  Z_{\sigma_*}  +\big(1-\hat U_{\sigma_*}^{-1}\big)\frac{\|\nnu\|_1}{\|\nnu\|_2}\big(1+O_P(N_{\sigma_*}^{-1/2})\big)\nonumber\\
    &=  Z_{\sigma_*}  +\bigg(\frac{\|1-\nnu\|_2}{\|1-\nnu\|_1}Z_{\sigma_*}'+ \frac{O_P(1)}{(p-N_{\sigma_*})(\sigma_*N_{\sigma_*})^2}\bigg)\frac{\|\nnu\|_1}{\|\nnu\|_2}\big(1+O_P(N_{\sigma_*}^{-1/2})\big)\nonumber\\
    &=  Z_{\sigma_*}  +\bigg(\frac{\|1-\nnu\|_2}{\|1-\nnu\|_1}Z_{\sigma_*}'+ \frac{O_P(1)}{(p-N_{\sigma_*})(\sigma_*N_{\sigma_*})^2} +
    \frac{O_P(1)}{(p-N_{\sigma_*})^{1/2}N_{\sigma_*}^{1/2}}\bigg)\frac{\|\nnu\|_1}{\|\nnu\|_2}\nonumber\\
    &=  Z_{\sigma_*}  +\frac{\|\nnu\|_1\|1-\nnu\|_2}{\|\nnu\|_2\|1-\nnu\|_1}Z_{\sigma_*}'+\frac{O_P(1)}{(p-N_{\sigma_*})\sigma_*^2N_{\sigma_*}^{3/2}}+\frac{O_P(1)}{(p-N_{\sigma_*})^{1/2}}.
\end{align*}
This result, combined with the obvious identity $|\eps_j|^2+|\eps^\diese_j|^2 = \frac12|{\eps_j+e^{-\ii j \bar\tau^*}\eps_j^\diese}|^2
+\frac12|{\eps_j-e^{-\ii j \bar\tau^*}\eps_j^\diese}|^2 = |\xi_j|^2+|\xi_j^\diese|^2$, completes the proof of the theorem.

\section{Bounds for the maxima of random sums}\label{sec:5}

In this section, we will gather some useful technical lemmas. They essentially characterize the stochastic behavior of
the maximum of the sum of independent random quantities, which are either ``simple'' Gaussian processes \citep{Berman88} or scaled
chi-squared random variables. Note that instead of Berman's formula, one could use the combination of the Dudley theorem and Talagrand's inequality  
to control the supremum of the considered random process. Both approaches lead to optimal results. However, the Berman formula is much easier
to apply since it avoids the computation of the covering numbers or other quantities of the same flavor. It should
be however emphasized that this is made possible by the fact that the noise is assumed Gaussian and the sample paths of the process to be maximized 
are continuously differentiable. In a more general situation (non Gaussian noise, nonsmooth sample paths, etc.) one would
most likely have no other choice than applying the Dudley-Talagrand routine.

\begin{proposition}[\cite{Berman88}]\label{prop:5}
Suppose that $g_j$ are continuously differentiable functions satisfying $\sum_{j=1}^n g_j(t)^2 = 1$ for all $t$, and $\xi_j \simiid \calN (0,1)$.
Then, for every $x>0$, we have
\begin{equation*}
\prob\bigg( \sup\limits_{[a,b]} \sum_{j=1}^n g_j(t) \xi_j \geq x\bigg) \leq \frac{L_0}{2\pi}e^{-\frac{x^2}{2}} + \int_x^{+\infty}
\frac{e^{-\frac{t^2}{2}}}{\sqrt{2\pi}}\,dt, \quad\text{with} \quad L_0 = \int_a^b \bigg[{\sum_{j=1}^n g_j'(t)^2}\bigg]^{1/2} \,dt.
\end{equation*}
\end{proposition}

We will also use the following fact about moderate deviations of the random variables that can be written as the sum of squares of independent centered Gaussian
random variables.

\begin{lemma}\label{lem:2}
Let $N$ be some positive integer and let $\etadiese_j$, $j=1,\ldots,N$ be independent complex valued random variables such that their real
and imaginary parts are independent standard Gaussian variables. Let $\bs=(s_1,\ldots,s_N)$ be a vector of real numbers. For any
$y\ge 0$, it holds that
\begin{equation*}
\prob\Big\{\sum_{j=1}^N s_j^2|\etadiese_j|^2\ge 2\|\bs\|_2^2+2\sqrt{2} \|\bs\|_4^2y+2\|\bs\|_\infty^2y^2\Big\}\le e^{-y^2/2},
\end{equation*}
with the standard notations $\|\bs\|_\infty=\max\limits_{j=1,...,N} |s_j|$ and $\|\bs\|_q^q=\sum_{j=1}^N |s_j|^q$.
\end{lemma}

\begin{proof}
This is a direct consequence of \cite[Lemma 1]{LaurentMassart}.
\end{proof}

\begin{proof}[Proof of Lemma~\ref{lem:2.5}]
We apply Proposition~\ref{prop:5} to the functions $\{g_j(t)\}_{j=1,\ldots,2N}$ defined on the interval $[a,b]=[0,2\pi]$ by
$g_j(t)=[s_j\cos(jt)\fcar(j\le N)+s_j\sin(jt)\fcar(j>N)]/\|\bs\|_2$. One easily checks that for every $t\in[0,2\pi]$,
we have $\sum_{j=1}^{2N} g_j^2(t)=1$. Therefore, applying Berman's result we get
$$
\prob\big(\|Z\|_\infty \ge \|\bs\|_2 x\big) = \prob\Big(\sup_{t\in[0,2\pi]} \sum_{j\le N} g_j(t)\xi_j+\sum_{j>N} g_j(t)\xi_j'\ge x\Big)\le
\Big(\frac{L_0}{2\pi}+1\Big)e^{-x^2/2}.
$$
In the present context, the constant $L_0$ admits the following simple upper bound:
$$
L_0=\int_0^{2\pi} \Big[\sum_{j=1}^N j^2 s_j^2/\|\bs\|_2^2 \Big]^{1/2}\,dt \le 2\pi N,
$$
which yields the desired result.
\end{proof}	

\begin{proof}[Proof of Lemma~\ref{lem:5}]
First note that we can not directly apply Berman's formula, since the summands are not Gaussian. However, they are conditionally Gaussian if the conditioning is done, for example, with respect to the sequence $\{\etadiese_j\}_{j=1,\ldots,N}$. Indeed, one easily checks that conditionally to $\{\etadiese_j\}_{j=1,\ldots,N}$, the random processes
\begin{equation*}
\tau\mapsto \sum\limits_{j=1}^{N}  s_j \Re\big(e^{\ii j\tau}\eta_j\etadiese_j\big)
\end{equation*}
and
$$
\tau\mapsto \sum\limits_{j=1}^{N}  s_j |\etadiese_j| \big( \cos(j\tau)\xi_j - \sin(j\tau) \xi'_j\big)\quad \text{ with $\xi_j,\xi'_j \simiid \Nzeroun$}
$$
have the same distributions.  Therefore, it follows from Lemma~\ref{lem:2.5} that
\begin{equation*}
\prob\bigg( \sup\limits_{[0,2\pi]} \Big|\sum_{j=1}^N s_j \Re\big(e^{\ii j\tau}\eta_j\etadiese_j\big)\Big| \geq x \Big(\sum_{j=1}^N s_j^2|\etadiese_j|^2\Big)^\frac{1}{2}\Big| \{\etadiese_j\}_{j=1,\ldots,N}\bigg) \leq (N+1)e^{-x^2/2}.
\end{equation*}
Let us now denote by $\zeta$ the square root of the random variable $\sum_{j=1}^N s_j^2|\etadiese_j|^2$.  It is clear that for all $a>0$,
\begin{align*}
\prob\big(\|Z\|_\infty\ge ax\big) &= \prob\big(\|Z\|_\infty\ge ax;\;\zeta \le a\big)+\prob\big(\|Z\|_\infty\ge ax;\;\zeta > a\big)\\
													&\le \prob\big(\|Z\|_\infty\ge x\zeta\big)+\prob\big(\zeta > a\big)\\
													&\le (N+1)e^{-x^2/2}+\prob\big(\zeta > a\big).
\end{align*}
To complete the proof, it suffices to replace $a$ by $\sqrt{2}(\|\bs\|_2+y\|\bs\|_\infty)$ and to apply Lemma~\ref{lem:2} along with the
inequalities $\|\bs\|_2+\|\bs\|_\infty y= (\|\bs\|_2^2+2\|\bs\|_\infty\|\bs\|_2y+\|\bs\|_\infty^2y^2)^{1/2}\ge (\|\bs\|_2^2+\sqrt{2}\|\bs\|_4^2y+\|\bs\|_\infty^2y^2)^{1/2}$.
\end{proof}


\section*{Acknowledgments} This work has been partially supported by the ANR grant Callisto and the
grant Investissements d'Avenir (ANR-11-IDEX-0003/Labex Ecodec/ANR-11-LABX-0047).

\section*{References}

\bibliography{bibliographie}

\begin{thebibliography}{44}
\expandafter\ifx\csname natexlab\endcsname\relax\def\natexlab#1{#1}\fi
\expandafter\ifx\csname url\endcsname\relax
  \def\url#1{\texttt{#1}}\fi
\expandafter\ifx\csname urlprefix\endcsname\relax\def\urlprefix{URL }\fi

\bibitem[{Baraud et~al.(2003)Baraud, Huet, and Laurent}]{Baraudetal03}
Baraud, Y., Huet, S., Laurent, B., 2003. Adaptive tests of linear hypotheses by
  model selection. Ann. Statist. 31~(1), 225--251.

\bibitem[{Baraud et~al.(2005)Baraud, Huet, and Laurent}]{Baraudetal05}
Baraud, Y., Huet, S., Laurent, B., 2005. Testing convex hypotheses on the mean
  of a {G}aussian vector. {A}pplication to testing qualitative hypotheses on a
  regression function. Ann. Statist. 33~(1), 214--257.

\bibitem[{Berman(1988)}]{Berman88}
Berman, S.~M., 1988. Sojourns and extremes of a stochastic process defined as a
  random linear combination of arbitrary functions. Comm. Statist. Stochastic
  Models 4~(1), 1--43.

\bibitem[{Bigot and Gadat(2010)}]{BigotGadat10}
Bigot, J., Gadat, S., 2010. A deconvolution approach to estimation of a common
  shape in a shifted curves model. Ann. Statist. 38~(4), 2422--2464.

\bibitem[{Bigot et~al.(2009{\natexlab{a}})Bigot, Gadat, and
  Loubes}]{Bigotetal09}
Bigot, J., Gadat, S., Loubes, J.-M., 2009{\natexlab{a}}. Statistical
  {M}-estimation and consistency in large deformable models for image warping.
  J. Math. Imaging Vision 34~(3), 270--290.

\bibitem[{Bigot et~al.(2009{\natexlab{b}})Bigot, Gamboa, and
  Vimond}]{BigotVim09}
Bigot, J., Gamboa, F., Vimond, M., 2009{\natexlab{b}}. Estimation of
  translation, rotation, and scaling between noisy images using the
  {F}ourier-{M}ellin transform. SIAM J. Imaging Sci. 2~(2), 614--645.

\bibitem[{Brown and Low(1996)}]{BrownLow96}
Brown, L.~D., Low, M.~G., 1996. Asymptotic equivalence of nonparametric
  regression and white noise. Ann. Statist. 24~(6), 2384--2398.

\bibitem[{Carroll and Hall(1992)}]{Carroll92}
Carroll, R.~J., Hall, P., 1992. Semiparametric comparison of regression curves
  via normal likelihoods. Austral. J. Statist. 34~(3), 471--487.

\bibitem[{Castillo(2007)}]{Castillo07}
Castillo, I., 2007. Semi-parametric second-order efficient estimation of the
  period of a signal. Bernoulli 13~(4), 910--932.

\bibitem[{Castillo(2012)}]{Castillo11}
Castillo, I., 2012. A semiparametric {B}ernstein-von {M}ises theorem for
  {G}aussian process priors. Probab. Theory Related Fields 152~(1-2), 53--99.

\bibitem[{Castillo and Loubes(2009)}]{Castillo09}
Castillo, I., Loubes, J.-M., 2009. Estimation of the distribution of random
  shifts deformation. Math. Methods Statist. 18~(1), 21--42.

\bibitem[{Collier(2012)}]{Collier2012}
Collier, O., 2012. Minimax hypothesis testing for curve registration. Electron.
  J. Stat. 6, 1129--1154.

\bibitem[{Comminges and Dalalyan(2013)}]{Comminges2013}
Comminges, L., Dalalyan, A.~S., 2013. Minimax testing of a composite null
  hypothesis defined via a quadratic functional in the model of regression.
  Electron. J. Stat. 7, 146--190.

\bibitem[{Dalalyan(2007)}]{Dalalyan08}
Dalalyan, A.~S., 2007. Stein shrinkage and second-order efficiency for
  semiparametric estimation of the shift. Math. Methods Statist. 16~(1),
  42--62.

\bibitem[{Dalalyan and Collier(2012)}]{jmlr_DalalyanC12}
Dalalyan, A.~S., Collier, O., 2012. Wilks' phenomenon and penalized
  likelihood-ratio test for nonparametric curve registration. In: Journal of
  Machine Learning Research - Proceedings Track 22. pp. 264--272.

\bibitem[{Dalalyan et~al.(2006)Dalalyan, Golubev, and Tsybakov}]{Dalalyan06}
Dalalyan, A.~S., Golubev, G.~K., Tsybakov, A.~B., 2006. Penalized maximum
  likelihood and semiparametric second-order efficiency. Ann. Statist. 34~(1),
  169--201.

\bibitem[{Dalalyan and Rei{\ss}(2007)}]{DR07}
Dalalyan, A.~S., Rei{\ss}, M., 2007. Asymptotic statistical equivalence for
  ergodic diffusions: the multidimensional case. Probab. Theory Related Fields
  137~(1-2), 25--47.

\bibitem[{Fan and Jiang(2007)}]{Fan07}
Fan, J., Jiang, J., 2007. Nonparametric inference with generalized likelihood
  ratio tests. TEST 16~(3), 409--444.

\bibitem[{Fan et~al.(2001)Fan, Zhang, and Zhang}]{FanZhang01}
Fan, J., Zhang, C., Zhang, J., 2001. Generalized likelihood ratio statistics
  and {W}ilks phenomenon. Ann. Statist. 29~(1), 153--193.

\bibitem[{Gamboa et~al.(2007)Gamboa, Loubes, and Maza}]{Gamboa07}
Gamboa, F., Loubes, J.-M., Maza, E., 2007. Semi-parametric estimation of
  shifts. Electron. J. Stat. 1, 616--640.

\bibitem[{Gayraud and Pouet(2005)}]{Pouet05}
Gayraud, G., Pouet, C., 2005. Adaptive minimax testing in the discrete
  regression scheme. Probab. Theory Related Fields 133~(4), 531--558.

\bibitem[{Glasbey and Mardia(2001)}]{Glasbey01}
Glasbey, C.~A., Mardia, K.~V., 2001. A penalized likelihood approach to image
  warping. J. R. Stat. Soc. Ser. B Stat. Methodol. 63~(3), 465--514.

\bibitem[{Golubev(1988)}]{Golubev88}
Golubev, G.~K., 1988. Estimation of the period of a signal with an unknown form
  against a white noise background. Problems Inform. Transmission 24~(4),
  288--299.

\bibitem[{Grama and Nussbaum(1998)}]{GramaNussbaum}
Grama, I., Nussbaum, M., 1998. Asymptotic equivalence for nonparametric
  generalized linear models. Probab. Theory Related Fields 111~(2), 167--214.

\bibitem[{H{\"a}rdle and Marron(1990)}]{Hardle90}
H{\"a}rdle, W., Marron, J.~S., 1990. Semiparametric comparison of regression
  curves. Ann. Statist. 18~(1), 63--89.

\bibitem[{Hartley and Zisserman(2003)}]{Hartley_Ziss}
Hartley, R., Zisserman, A., 2003. Multiple view geometry in computer vision,
  2nd Edition. Cambridge University Press, Cambridge.

\bibitem[{Horowitz and Spokoiny(2001)}]{Horowitz01}
Horowitz, J.~L., Spokoiny, V.~G., 2001. An adaptive, rate-optimal test of a
  parametric mean-regression model against a nonparametric alternative.
  Econometrica 69~(3), 599--631.

\bibitem[{Immerkaer(1996)}]{Immerkaer}
Immerkaer, J., 1996. Fast noise variance estimation. Computer Vision and Image
  Understanding 64~(2), 300 -- 302.

\bibitem[{Ingster and Kutoyants(2007)}]{Ingster_Kutoyants}
Ingster, Y.~I., Kutoyants, Y.~A., 2007. Nonparametric hypothesis testing for
  intensity of the {P}oisson process. Math. Methods Statist. 16~(3), 217--245.

\bibitem[{Ingster and Suslina(2003)}]{Ingster_Suslina}
Ingster, Y.~I., Suslina, I.~A., 2003. Nonparametric goodness-of-fit testing
  under {G}aussian models. Vol. 169 of Lecture Notes in Statistics.
  Springer-Verlag, New York.

\bibitem[{Kneip and Gasser(1992)}]{Kneip92}
Kneip, A., Gasser, T., 1992. Statistical tools to analyze data representing a
  sample of curves. Ann. Statist. 20~(3), 1266--1305.

\bibitem[{Laurent and Massart.(2000)}]{LaurentMassart}
Laurent, B., Massart., P., 2000. Adaptive estimation of a quadratic functional
  by model selection. Ann. Statist. 28~(5), 1302--1338.

\bibitem[{Loubes et~al.(2006)Loubes, Maza, Lavielle, and
  Rodr{\'{\i}}guez}]{Loubes06}
Loubes, J.-M., Maza, E., Lavielle, M., Rodr{\'{\i}}guez, L., 2006. Road
  trafficking description and short term travel time forecasting, with a
  classification method. Canad. J. Statist. 34~(3), 475--491.

\bibitem[{Lowe(2004)}]{Lowe}
Lowe, D.~G., 2004. {Distinctive image features from scale-invariant keypoints}.
  International journal of computer vision 60~(2), 91--110.

\bibitem[{Munk and Dette(1998)}]{Munk_Dette}
Munk, A., Dette, H., 1998. Nonparametric comparison of several regression
  functions: exact and asymptotic theory. Ann. Statist. 26~(6), 2339--2368.

\bibitem[{Neumeyer and Dette(2003)}]{Neumeyer}
Neumeyer, N., Dette, H., 2003. Nonparametric comparison of regression curves:
  an empirical process approach. Ann. Statist. 31~(3), 880--920.

\bibitem[{Nussbaum(1996)}]{Nussbaum96}
Nussbaum, M., 1996. Asymptotic equivalence of density estimation and {G}aussian
  white noise. Ann. Statist. 24~(6), 2399--2430.

\bibitem[{Pardo-Fern{\'a}ndez et~al.(2007)Pardo-Fern{\'a}ndez, Van~Keilegom,
  and Gonz{\'a}lez-Manteiga}]{PardoFernandez}
Pardo-Fern{\'a}ndez, J.~C., Van~Keilegom, I., Gonz{\'a}lez-Manteiga, W., 2007.
  Testing for the equality of {$k$} regression curves. Statist. Sinica 17~(3),
  1115--1137.

\bibitem[{Petrov(1995)}]{Petrov}
Petrov, V.~V., 1995. Limit theorems of probability theory. Vol.~4 of Oxford
  Studies in Probability. The Clarendon Press Oxford University Press, New
  York.

\bibitem[{Rei{\ss}(2008)}]{Reiss}
Rei{\ss}, M., 2008. Asymptotic equivalence for nonparametric regression with
  multivariate and random design. Ann. Statist. 36~(4), 1957--1982.

\bibitem[{R{\o}nn(2001)}]{Ronn01}
R{\o}nn, B.~B., 2001. Nonparametric maximum likelihood estimation for shifted
  curves. J. R. Stat. Soc. Ser. B Stat. Methodol. 63~(2), 243--259.

\bibitem[{Srihera and Stute(2010)}]{Srihera}
Srihera, R., Stute, W., 2010. Nonparametric comparison of regression functions.
  J. Multivariate Anal. 101~(9), 2039--2059.

\bibitem[{Trigano et~al.(2011)Trigano, Isserles, and Ritov}]{TriganoRitov}
Trigano, T., Isserles, U., Ritov, Y., 2011. Semiparametric curve alignment and
  shift density estimation for biological data. IEEE Transactions on Signal
  Processing 59~(5), 1970--1984.

\bibitem[{Vimond(2010)}]{Vimond10}
Vimond, M., 2010. Efficient estimation for a subclass of shape invariant
  models. Ann. Statist. 38~(3), 1885--1912.

\end{thebibliography}

\end{document}